\documentclass[12pt]{amsart}
\usepackage{amssymb}
\usepackage{datetime,wasysym}

\newtheoremstyle{slthm}
{9pt}
{5pt}
{\slshape}
{}
{\bfseries}
{.}
{.5em}
{\thmname{#1}\thmnumber{ #2}\thmnote{ (#3)}}

\newtheoremstyle{prcl}
{9pt}
{5pt}
{\slshape}
{}
{\bfseries}
{.}
{.5em}
{\thmname{#3}\thmnumber{ #2}}

\newtheoremstyle{prblm}
{9pt}
{5pt}
{\rm}
{}
{\bfseries}
{.}
{.5em}
{\thmname{#3}\thmnumber{ #2}}

\theoremstyle{slthm}
\newtheorem{thm}{Theorem}[section]

\newtheorem{lemma}[thm]{Lemma}
\newtheorem{prop}[thm]{Proposition}
\newtheorem{cor}[thm]{Corollary}

\newtheorem{facts}[thm]{Facts}

\theoremstyle{definition}

\newtheorem{df}[thm]{Definition}

\newtheorem{nrmk}[thm]{Remark}
\newtheorem{nrmks}[thm]{Remarks}
\newtheorem{expl}[thm]{Example}
\newtheorem{expls}[thm]{Examples}

\theoremstyle{remark}
\newtheorem*{rmk}{Remark}
\newtheorem*{rmks}{Remarks}

\theoremstyle{prcl}

\newtheorem{nprclaim}[thm]{Proclaim}

\theoremstyle{prblm}

\newenvironment{renumerate}
{\renewcommand{\labelenumi}{\rm (\roman{enumi})}
	\begin{enumerate}}
	{\end{enumerate}}

\newcounter{flexnummark}

\DeclareMathOperator{\eh}{eh}
\DeclareMathOperator{\lm}{lm}

\DeclareMathOperator{\el}{el}

\DeclareMathOperator{\level}{level}

\DeclareMathOperator{\alevel}{al}
\DeclareMathOperator{\Log}{Log}
\DeclareMathOperator{\Exp}{Exp}
\DeclareMathOperator{\llog}{\mathfrak{log}}
\DeclareMathOperator{\eexp}{\mathfrak{exp}}

\DeclareMathOperator{\nn}{\mathfrak{n}}
\DeclareMathOperator{\vv}{\mathfrak{v}}

\DeclareMathOperator{\uu}{\mathfrak{u}}

\DeclareMathOperator{\Hsr}{\H_{st}}

\DeclareMathOperator{\cl}{cl}

\DeclareMathOperator{\im}{Im}
\DeclareMathOperator{\re}{Re}
\DeclareMathOperator{\supp}{supp}

\DeclareMathOperator{\sign}{sgn}

\DeclareMathOperator\h{h}

\DeclareMathOperator{\bexp}{\textbf{exp}}
\DeclareMathOperator{\blog}{\textbf{log}}

\newcommand{\rest}[1]{\!\!\upharpoonright_{#1}}

\newcommand{\into}{\longrightarrow}

\renewcommand{\bar}{\overline}

\def\Ind#1#2{#1\setbox0=\hbox{$#1x$}\kern\wd0\hbox to 0pt{\hss$#1\mid$\hss}
	\lower.9\ht0\hbox to 0pt{\hss$#1\smile$\hss}\kern\wd0}

\def\Notind#1#2{#1\setbox0=\hbox{$#1x$}\kern\wd0\hbox to 0pt{\mathchardef
		\nn=12854\hss$#1\nn$\kern1.4\wd0\hss}\hbox to
	0pt{\hss$#1\mid$\hss}\lower.9\ht0 \hbox to
	0pt{\hss$#1\smile$\hss}\kern\wd0}

\newcommand{\set}[1]{\left\{#1\right\}}

\newcommand{\KK}{\mathbb{K}}
\newcommand{\LL}{\mathbb{L}}
\newcommand{\NN}{\mathbb{N}}
\newcommand{\ZZ}{\mathbb{Z}}

\newcommand{\RR}{\mathbb{R}}

\newcommand{\CC}{\mathbb{C}}

\newcommand{\curly}[1]{\mathcal{#1}}
\newcommand{\A}{\curly{A}}
\newcommand{\B}{\curly{B}}
\newcommand{\C}{\curly{C}}
\newcommand{\D}{\curly{D}}
\newcommand{\E}{\curly{E}}
\newcommand{\F}{\curly{F}}
\newcommand{\G}{\curly{G}}
\renewcommand{\H}{\curly{H}}
\newcommand{\I}{\curly{I}}

\newcommand{\K}{\curly{K}}

\newcommand{\M}{\curly{M}}

\renewcommand{\P}{\curly{P}}

\newcommand{\R}{\curly{R}}
\renewcommand{\S}{\curly{S}}

\newcommand{\U}{\curly{U}}

\newcommand{\f}{\mathfrak{f}}
\newcommand{\g}{\mathfrak{g}}
\renewcommand{\h}{\mathfrak h}

\newcommand{\m}{\mathfrak{m}}
\newcommand{\n}{\mathfrak{n}}
\newcommand{\p}{\mathfrak{p}}

\renewcommand{\t}{\mathfrak{t}}
\renewcommand{\u}{\mathfrak{u}}
\renewcommand{\v}{\mathfrak{v}}

\newcommand{\frA}{\mathfrak{A}}
\newcommand{\frB}{\mathfrak{B}}

\newcommand{\frM}{\mathfrak{M}}

\newcommand{\frS}{\mathfrak{S}}
\newcommand{\frU}{\mathfrak{U}}
\newcommand{\frV}{\mathfrak{V}}

\newcommand{\ff}{\mathbf{f}}
\renewcommand{\gg}{\mathbf{g}}

\renewcommand{\nn}{\mathbf{n}}

\renewcommand{\ss}{\mathbf{s}}

\renewcommand{\uu}{\mathbf{u}}
\renewcommand{\vv}{\mathbf{v}}

\newcommand{\bP}{\mathbf{P}}

\newcommand{\la}{\curly{L}}

\DeclareMathOperator{\Ranexp}{\RR_{an,exp}}

\DeclareMathOperator{\Lanexp}{\la_{an,exp}}
\DeclareMathOperator{\Lanexplog}{\la_{an,exp,log}}

\newcommand{\hplus}{\H^{>0}}
\newcommand{\bhplus}{\bar{\H^{>0}}}
\newcommand{\Hang}{\H^{\text{ang}}}
\newcommand{\Hpoly}{\H_{\text{poly}}}
\newcommand{\Rpoly}{\R_{\text{poly}}}
\newcommand{\Mpoly}{\frM_{\text{poly}}}
\newcommand{\Hone}{\H^{\asymp 1}}

\newcommand{\Pc}[2]{\mathbb{#1}\left\{#2\right\}}

\allowdisplaybreaks[2]

\numberwithin{equation}{section}

\title {Analytic continuation of $\log$-$\exp$-analytic germs}

\author {Tobias Kaiser and Patrick Speissegger}

\address{Universit\"at Passau \\
Fakult\"at f\"ur Informatik und Mathematik \\
Innstr. 33 \\
94032 Passau \\
Germany}
\email{tobias.kaiser@uni-passau.de}

\address {Department of Mathematics and Statistics, McMaster University, 1280
Main Street West, Hamilton, Ontario L8S 4K1, Canada}

\email {speisseg@math.mcmaster.ca}

\date{\today\ at \currenttime}



\thanks{Second author supported by NSERC of Canada grant
RGPIN 261961 and the Zukunftskolleg of Universit\"at Konstanz}

\begin{document}

\begin{abstract}
	We describe maximal, in a sense made precise, $\LL$-analytic continuations of germs at $+\infty$ of unary functions definable in the o-minimal structure $\Ranexp$ on the Riemann surface $\LL$ of the logarithm.  As one application, we give an upper bound on the logarithmic-exponential complexity of the compositional inverse of an infinitely increasing such germ, in terms of its own logarithmic-exponential complexity and its level.  As a second application, we strengthen Wilkie's theorem on definable complex analytic continuations of germs belonging to the residue field $\Rpoly$ of the valuation ring of all polynomially bounded definable germs.
\end{abstract}

\maketitle
\markboth{Tobias Kaiser and Patrick Speissegger}{Analytic continuation of $\log$-$\exp$-analytic germs}

\section*{Introduction}

The o-minimal structure $\Ranexp$, see van den Dries and Miller \cite{Dries:1994eq} or van den Dries, Macintyre and Marker \cite{Dries:1994tw}, is one of the most important regarding applications, because it defines all elementary functions (with the necessary restriction on periodic ones such as $\sin$ or $\cos$).  Holomorphic functions definable in $\Ranexp$ have turned out to be crucial in applications to diophantine geometry, see for instance Pila \cite{MR2800724} and Peterzil and Starchenko \cite{MR3039679}.

It is known \cite{Dries:1994eq,Dries:1994tw} that every function definable in the o-minimal structure $\Ranexp$ is piecewise analytic.  This implies that, if $f$ is the germ at $+\infty$ of a one-variable function definable in $\Ranexp$, also called a \textbf{$\log$-$\exp$-analytic germ} here, there is an open domain $U \subseteq \CC$ and a complex analytic continuation $\ff:U \into \CC$ of $f$, or an open domain $\frU \subseteq \LL$ and an $\LL$-analytic continuation $\f:\frU \into \LL$, where $\LL$ is the Riemann surface of the logarithm (see Section \ref{real_domains} for details).  Concerning complex analytic continuations $\ff:U \into \CC$ of $f$, it is shown by Kaiser \cite[Theorem C]{MR3509471} that $\ff$ can be chosen to be definable.  Wilkie \cite[Theorem 1.11]{MR3509953} characterizes those $f$ for which $\ff$ extends definably on some right translate of a sector properly containing a right half-plane of $\CC$; he then applies this continuation result to a diophantine problem.  

The aim of this paper is to describe $\LL$-analytic continuations $\f:\frU \into \LL$: we find a maximal $\frU$ (in a sense to be made precise) such that $\f$ is \textbf{half-bounded}, that is, either $\f$ or $1/\f$ is bounded (see the Continuation Corollary, Theorem \ref{H_ext_cor} below).  We obtain this statement from the more precise Continuation Theorem \ref{I_ext_prop} below, which only applies to \textbf{infinitely increasing} $f$, that is, those $f$ for which $\lim_{x \to +\infty} f(x) = +\infty$ holds.

These $\LL$-analytic continuations of $f$ depend on two integer-valued quantities associated to $f$: the \textit{exponential height} $\eh(f)$ of $f$ and the \textit{level} $\level(f)$ of $f$.  The former measures the logarithmic-exponential complexity of $f$; roughly speaking, if $f$ is unbounded, then $\eh(\exp \circ f) = \eh(f)+1$, while if $f$ is bounded, then $\eh(\exp \circ f) = \eh(f)$ (see Section \ref{description} for details).  The latter measures the exponential order of growth of the germ $f$; we refer the reader to Marker and Miller \cite{Marker:1997kn} for details and to Facts \ref{level_facts} to recall the main properties.  The level extends to all $\log$-$\exp$-analytic germs in an obvious manner, see Section \ref{angular_level} below.  

\begin{rmk}
	We show in Section \ref{description} that $\level(f) \le \eh(f)$, for all $\log$-$\exp$-analytic germs $f$.  The two are not equal in general: we have $\level(x+e^{-x}) = 0 \ne 1 = \eh(x+e^{-x})$.  
\end{rmk}

What we find in the Continuation Corollary is that, if $\f:\frU \into \LL$ is a maximal, half-bounded $\LL$-analytic continuation of $f$, then the size (in a sense to be made precise) of $\frU$ is determined by $\eh(f)$ and, conversely, that the size of $\frU$ determines an upper bound on $\eh(f)$.  Moreover, if $f$ is infinitely increasing, we also find that $\f$ is injective and, in this case, $\level(f)$ determines the size of the image $\f(\frU)$; see the Simplified Continuation Theorem \ref{simplified_ext} below.  The Continuation Theorem \ref{I_ext_prop} is more technical, but it is the central result of this paper, as our applications actually rely on these extra technicalities.  Finally we also describe, in the Complex Continuation Corollary \ref{complex_cont_cor} below, the resulting maximal complex continuations of germs in $\H$.

We include two applications of the Continuation Theorem and its corollaries.  In Application \ref{eh_application}, we give an upper bound on $\eh(f^{-1})$, in terms of $\eh(f)$ and $\level(f)$, of an infinitely increasing $\log$-$\exp$-analytic germ $f$, where $f^{-1}$ denotes the compositional inverse of $f$.  In Application \ref{pre-wilkie}, we strengthen Wilkie's theorem \cite[Theorem 1.11]{MR3509953} on definable complex analytic continuations of germs belonging to the residue field $\Rpoly$ of the valuation ring of all polynomially bounded $\log$-$\exp$-analytic germs.

The main motivation for us to prove the Continuation Theorem, however, is to show that all \textit{principal monomials} of $\H$, as defined towards the end of Section \ref{description} below, can be used in asymptotic expansions to obtain a quasianalytic Ilyashenko field $\K$ extending the Ilyashenko field $\F$ constructed in Speissegger \cite{Speissegger:2016aa}.  The details of this application (which also relies on Application \ref{eh_application} below), its motivations and the construction of $\K$ are the subject of a forthcoming paper.

The paper is organized as follows: in Section \ref{main_section}, we introduce some of the terminology needed and state the main results (except for Continuation Theorem \ref{I_ext_prop}).  In Section \ref{description}, we give a description of the set $\H$ of all $\log$-$\exp$-analytic germs and introduce the formal equivalent of \textit{convergent LE-series}, in the spirit of van den Dries, Macintyre and Marker \cite[Remark 6.31]{Dries:2001kh}.  We introduce the notions needed  for the Continuation Theorem \ref{I_ext_prop} in Sections \ref{real_domains}, \ref{angular_level} and \ref{opposing} and give its proof in Sections \ref{extendable} and \ref{expansive_section}.  Applications \ref{eh_application} and \ref{pre-wilkie} are discussed in Sections \ref{geom_simple_section} and \ref{definability_section}, respectively.

\section{Statements of results and some ideas}
\label{main_section}

One of the main issues is that the language needed to state our results has to be developed from scratch.  To illustrate some of the notions involved, we now briefly describe that of an \textit{$\eta$-domain}, the type of domain in $\LL$ that the sets $\frU$ above are selected from, and we state the precise Continuation Corollary and a simplified version of the Continuation Theorem that we call the Simplified Continuation Theorem below; the full statement of the Continuation Theorem is deferred to Section \ref{expansive_section}.  

We denote by $\LL:= \set{(r,\theta):\ r>0,\ \theta \in \RR}$ the Riemann surface of the logarithm with its usual covering map $\pi:\LL \into \CC \setminus \{0\}$ defined by $\pi(r,\theta) = re^{i\theta}$.  We let $|(r,\theta)|:= r$ be the \textbf{modulus} and $\arg(r,\theta):= \theta$ be the \textbf{argument} of $(r,\theta)$.  We usually write $x = (|x|,\arg x)$ for an element of $\LL$, and we identify the positive real half-line $(0,+\infty)$ with the set $\{x \in \LL:\ \arg x = 0\}$.

To define an $\eta$-domain, we first introduce \textbf{real domains} in Section \ref{real_domains} as the sets of the form $$\frU_h:= \set{x \in H_\LL(a):\ |\arg x| < h(|x|)},$$ where $a \ge 0$, $$H_\LL(a):= \set{x \in \LL:\ |x| > a}$$ and $h:(a,+\infty) \into (0,+\infty)$ is continuous.  Identifying $\LL$ with the set $(0,+\infty) \times \RR$, a real domain $\frU_h$ is definable in $\Ranexp$ if and only if $h$ is a $\log$-$\exp$-analytic germ.  Considering two real domains $\frU_{h_1}$ and $\frU_{h_2}$ equivalent if there exists $a>0$ such that $\frU_{h_1} \cap H_\LL(a) = \frU_{h_2} \cap H_\LL(a)$, and calling the corresponding equivalence classes of real domains \textbf{germs at $\infty$ of real domains}, we get a bijective map $h \mapsto \frU_h:\H^{>0} \into \G^\infty_{\text{dfr}}(\LL)$, where $\H^{>0}$ is the set of all positive $\log$-$\exp$-analytic germs and $\G^\infty_{\text{dfr}}(\LL)$ denotes the set of all germs at $\infty$ of definable real domains.   This bijection satisfies $h_1 < h_2$ if and only if $\frU_{h_1} \subset \frU_{h_2}$ as germs at $\infty$; in particular, any measure of size on $\H^{>0}$, such as valuation or level, can be transferred to $\G^\infty_{\text{dfr}}(\LL)$.

To understand what measure of size on $\H^{>0}$ is appropriate for our purposes, we consider the analytic continuations on $\LL$ of the elementary functions (Section 2): scalar multiplication $$m_r(x):= rx$$ and the power function $$p_r(x):= x^r,$$ for $r>0$ and $x>0$, as well as $\exp$ and $\log$.  It is easy to see that $m_r$ and $p_r$ have definable, biholomorphic continuations $\m_r, \p_r:\LL \into \LL$, respectively, while $\exp$ has a holomorphic continuation $\eexp:\LL \into \LL$ that is neither definable nor injective.  However, $\log$ has a definable, biholomorphic continuation $\llog:H_\LL(1) \into S_\LL(\pi/2)$, where $$S_\LL(a):= \set{x \in \LL:\ |\arg x| < a}$$ for $a>0$.  Indeed, one of our reasons for working with $\LL$-analytic rather than complex analytic continuations is that the definable, biholomorphic restriction of $\eexp$ to $S_\LL(\pi/2)$ has a larger domain than any definable, biholomorphic complex continuation of $\exp$ in the right complex half-plane.

It is not hard to show (see Section \ref{real_domains}) that each of these biholomorphic continuations maps definable real domains to definable real domains.  But for any $f \in \hplus$, if $g \in \hplus$ is such that $\frU_g = \llog(\frU_f)$, then $g$ is bounded, as $\frU_g$ is a subset of $S_\LL(\pi/2)$ (as germs at $\infty$).  This ``big crunch'' indicates that none of the measures of size mentioned earlier are quite right to describe the behaviour of $\llog$.  We show in Section \ref{angular_level} (especially Proposition \ref{alevel_cor})  that there is a decreasing map $\alevel:\H^{>0} \into \NN \cup \{-1\}$, called \textbf{angular level}, such that the following hold for $f \in \H^{>0}$ and $r>0$:
\begin{renumerate}
	\item if $g \in \hplus$ is such that $\frU_g = \m_r(\frU_f)$, then $\alevel(g) = \alevel(f)$;
	\item if $g \in \hplus$ is such that $\frU_g = \p_r(\frU_f)$, then $\alevel(g) = \alevel(f)$;
	\item if $g \in \hplus$ is such that $\frU_g = \llog(\frU_f)$, then $\alevel(g) = \alevel(f) + 1$;
	\item if $f \le 1/\log$, then $\alevel(f) = \level(f) + 1$.
\end{renumerate}
This angular level is related to the usual level, as point (iv) indicates, but takes into account the ``big crunch'' of $\llog$ mentioned earlier.

Finally, not all $f \in \hplus$ are as well behaved as the elementary ones above: if $t_a(x) := x+a$ for $a>0$ and $x>0$, then $t_a$ has an injective, holomorphic continuation $\t_a:\LL \into \LL$ that is periodic in $\arg x$, hence is not definable and, in general, maps definable real domains to domains that are neither definable nor real.  However, it is easy to see that, given $f \in \hplus$, there are $g_1, g_2 \in \hplus$ such that $\alevel(f) = \alevel(g_1) = \alevel(g_2)$ and $\frU_{g_1} \subseteq \t_a(\frU_f) \subseteq \frU_{g_2}$.  This leads to our desired definition: given a domain $\frU \subseteq \LL$ and $\eta \in \NN \cup \{-1\}$, we call $\frU$ an \textbf{$\eta$-domain} if there exist $g_1, g_2 \in \hplus$ such that $\alevel(g_1) = \alevel(g_2) = \eta$ and $$\frU_{g_1} \subseteq \frU \subseteq \frU_{g_2}.$$  Every definable, real domain is an $\eta$-domain, for some appropriate $\eta$, and the maps $\m_r$ and $\p_r$, as well as $\t_a$, map $\eta$-domains to $\eta$-domains, while $\llog$ maps $\eta$-domains to $(\eta+1)$-domains.  

Denoting by $\H$ the set of all $\log$-$\exp$-analytic germs, we are now ready to state the Simplified Continuation Theorem, which describes biholomorphic continuations of all infinitely increasing $\log$-$\exp$-analytic germs in the spirit of those of the elementary germs described above.  A map $\varphi:\frU \into \LL$, with $\frU \subseteq \LL$, is called \textbf{angle-positive} if $$\sign(\arg\varphi(x)) = \sign(\arg x)$$ for all $x \in \frU$.

\begin{thm}[Simplified Continuation]
	\label{simplified_ext}
	Let $f \in \H$ be infinitely increasing, and set $\eta:= \max\{0,\eh(f)\}$ and $\lambda:= \level(f)$.  Then there exist an $(\eta-1)$-domain $\frU$, an $(\eta-1-\lambda)$-domain $\frV$ and an angle-positive, biholomorphic continuation $\f:\frU \into \frV$ of $f$ that maps $k$-domains to $(k-\lambda)$-domains, for $k \ge \eta-1$.
\end{thm}

This theorem was what we wanted to prove originally, but we were unable to do so without making the statement considerably more precise.  Therefore, we defer to the end of Section \ref{angular_level} to give an outline of the proof of the Continuation Theorems (simplified or not), where enough additional terminology is available to do so.

For arbitrary germs in $\H$, we have the following consequence of the Continuation Theorem:

\begin{cor}[Continuation]
	\label{H_ext_cor}
	Let $f \in \H$ and $\eta \in \NN$.  Then $\eh(f) \le \eta$ if and only if there exist an $(\eta-1)$-domain $\frU$ and a half-bounded, analytic continuation $\f:\frU \into \KK$ of $f$, where $\KK$ is either $\CC$ or $\LL$.
\end{cor}

The left-to-right implication of this corollary is a straightforward consequence of the Continuation Theorem (see Corollary \ref{ltr-implication}).  We obtain the right-to-left implication of the Continuation Corollary by combining the Continuation Theorem with a consequence of the Phrag\-m\'en-Lindel\"of principle found in Ilyashenko and Yakovenko \cite[Lemma 24.37]{Ilyashenko:2008fk}, see Proposition \ref{max_eh} below.

We now give two applications of the Continuation Theorem and its corollary.  The first of these considers the following question: if $f$ is infinitely increasing, and if $f^{-1}$ denotes the compositional inverse of $f$, is there a bound on $\eh\left(f^{-1}\right)$ described in terms of $\eh(f)$?  (Note that this question with level in place of exponential height has the natural answer $\level(f^{-1}) = -\level(f)$, see \cite{Marker:1997kn}.)  What we find (see Corollary \ref{gp_ext}(1) below) is the following more general statement:

\begin{nprclaim}[Application] 
	\label{eh_application}
	Assume that $f$ is infinitely increasing.  Then $$\eh\left(g \circ f^{-1}\right) \le \max\{\eh(g) + \eh(f) - 2\level(f), \eh(f) - \level(f)\},$$ where $g$ is any $\log$-$\exp$-analytic germ.
\end{nprclaim}

In Section \ref{geom_simple_section}, we also call a $\log$-$\exp$-analytic germ \textbf{simple} if $\eh(f) = \level(f)$, and we establish several consequences of Application \ref{eh_application} for such germs. 

The second application concerns the definability of the analytic continuations obtained in the Continuation Theorem and its corollary.  Wilkie \cite[Theorem 1.11]{MR3509953} obtains complex definable continuations for the germs contained in the subset $\Rpoly$ of $\H$ (see Section \ref{definability_section}), in fact characterizing $\Rpoly$ as the set of all $f \in \H$ that have a definable, complex continuation $\ff$ on some right translate of a sector properly containing a right half-plane of $\CC$.  

Building on the Continuation Theorem \ref{I_ext_prop}, we determine exactly which restrictions of our maximal $\LL$-analytic continuations are definable:  we call a set $\frS \subseteq \LL$ \textbf{angle-bounded} if the set $\{|\arg x|:\ x \in \frS\}$ is bounded.

\begin{thm}[Definability]
	Let $f \in \H$ be infinitely increasing, and set $\eta:= \max\{0,\eh(f)\}$ and $\lambda:= \level(f)$.  Let $\f:\frU \into \frV$ be one of the biholomorphic continuations of $f$ obtained from the Simplified Continuation Theorem, and let $\frU' \subseteq \frU$ be a definable domain.  If $\f(\frU')$ is angle-bounded, then $\f\rest{\frU'}$ is definable. 
\end{thm}
 
As a consequence, we obtain a similar (but not identical) characterization of membership in $\Rpoly$ that strengthens the complex continuation part of \cite[Theorem 1.11]{MR3509953}:

\begin{nprclaim}[Application]
	\label{pre-wilkie}
	Let $f \in \H$.  Then $f \in \Rpoly$ if and only if there exist a $(-1)$-domain $\frU$ and a half-bounded, analytic continuation $\ff:\frU \into \CC$ of $f$ such that, for every angle-bounded, definable domain $\frU' \subseteq \frU$, the restriction $\ff\rest{\frU'}$ is definable.
\end{nprclaim}

We note that the continuations in Application \ref{pre-wilkie} are complex-valued analytic continuations on domains in $\LL$.

Finally, in the Complex Continuation Corollary \ref{complex_cont_cor} below, we give a description of \textit{complex} analytic continuations of the germs in $\H$ implied by the Continuation Theorem and Corollary.  For $a \ge 0$, we set $$H(a):= \set{z \in \CC:\ \re z > a}.$$  We denote by $\arg$ the standard branch of the argument on $\CC \setminus (-\infty,0]$ and, for $\alpha \in (0,\pi]$, we set $$S(\alpha):= \set{z \in \CC:\ |\arg z| < \alpha}.$$  The following special case of Corollary \ref{complex_cont_cor} is worth writing down, as it avoids all $\LL$-related terminology introduced earlier:  

\begin{cor}
	\label{ccc}
	Let $f \in \H$ be such that $\eh(f) \le 0$.
	\begin{enumerate}
		\item There are $a \ge 0$ and a half-bounded complex analytic continuation $\ff:H(a) \into \CC$ of $f$.   
		\item Assume in addition that $f \in \I$.  Then 
		\begin{enumerate}
			\item $|\ff(z)| \to \infty$ as $|z| \to \infty$, for $z \in H(a)$;
			\item if $f \prec x^2$, then $\ff(H(a)) \subseteq \CC \setminus (-\infty,0]$, $\ff:H(a) \into \ff(H(a))$ is biholomorphic and we have $$\sign(\arg\ff(z)) = \sign(\arg z) = \sign(\im z) =
			\sign(\im\ff(z))$$ for $z \in H(a)$;
			\item if $\eh(f) < 0$ then, for every $\alpha > 0$, there exists $b \ge a$ such that $\ff(H(a)) \cap H(b) \subseteq S(\alpha)$. \qed
		\end{enumerate}
	\end{enumerate}
\end{cor}

Indeed, in our forthcoming paper generalizing the construction of Ilyashenko algebras in \cite{Speissegger:2016aa}, the only results we need from this paper are Corollary \ref{ccc} and Application \ref{eh_application}.

\section{A description of the Hardy field of $\Ranexp$}
\label{description}

One of the main results of \cite{Dries:1994tw} is that every function definable in $\Ranexp$ is piecewise given by $\Lanexplog$-terms.  Working from this result, the goal of this section is to describe the set $\H$ of all germs at $+\infty$ of unary functions definable in $\Ranexp$ in the spirit of the LE-series in \cite{Dries:2001kh}.

We let $\C$ be the ring of all germs at $+\infty$ of continuous functions $f:\RR \into \RR$ (so that $\H \subset \C$).  
A germ $f \in \C$ is $$\text{\textbf{small} if } \lim_{x \to +\infty} f(x) = 0,$$ $$\text{\textbf{large} if } \lim_{x \to +\infty} |f(x)| = \infty$$ and $$\text{\textbf{infinitely increasing} if } \lim_{x \to +\infty} f(x) = +\infty.$$  To compare elements of $\C$, we use the dominance relation $\prec$ found in Aschenbrenner and van den Dries \cite[Section 1]{MR2130825}, defined by $f \prec g$ if and only if $g(x) \ne 0$ for all sufficiently large $x$ and $\lim_{x \to +\infty} f(x)/g(x) = 0$ or, equivalently, $f(x) = o(g(x))$ as $x \to +\infty$.  Thus, $f \preceq g$ if and only if $f(x) = O(g(x))$ as $x \to +\infty$, and we write $f \asymp g$ if and only if $f \preceq g$ and $g \preceq f$.  Note that the relation $\asymp$ is an equivalence relation on $\C$, and the corresponding equivalence classes are the \textbf{Archimedean classes} of $\C$; we denote by $\Pi_\asymp:\C \into \C/_\asymp$ the corresponding projection map.

For $h \in \H$, we set
$$h(+\infty) := \lim_{x \to +\infty} h(x) \in \RR \cup \{-\infty,+\infty\},$$   and we denote by $\I$ the set of all infinitely increasing germs in $\H$, that is,  $$\I:= \set{f \in \H:\ f(+\infty) = +\infty}.$$  Note that $\I$ is a group under composition.

We start by defining the exponential height for germs defined by $\Lanexp$-terms.  To do so, we define the set $\E \subseteq \H$ of \textbf{$\Lanexp$-germs} and  associated \textbf{exponential height} $\eh:\E \into \NN \cup \{-\infty\}$ and \textbf{exponential level} $\el:\E \into \NN \cup \{-\infty\}$.  More precisely, we fix $n \in \NN \cup \{-\infty\}$ and define, by induction on $n$, the following sets:
\begin{itemize}
\item the set $\P\M_n$ of \textbf{pure exponential monomials} of exponential height $n$;
\item the set $\M_n \supseteq \P\M_n$ of \textbf{exponential monomials} of exponential height at most $n$;
\item the set $\E_n \supseteq \M_n$ of \textbf{$\Lanexp$-germs} of exponential height at most $n$;
\item the set $\P\E_n \subseteq \E_n$ of \textbf{pure} $\Lanexp$-germs of exponential height $n$, and its subset $\P\E_n^\infty$ of all \textbf{purely infinite} $\Lanexp$-germs;
\end{itemize}
such that the following hold:
\begin{itemize}
\item[(E1)$_n$] $\P\M_n \cup \{1\}$ is a multiplicative subgroup of $\H$, $$\P\M_n = \exp \circ \P\E_{n-1}^\infty \quad\text{if } n \ge 1,$$ the sets $\P\M_{-\infty}, \P\M_0, \dots, \P\M_n$ are pairwise disjoint and, for $m \in \P\M_n$ and $n \ge 0$, there exists a nonzero $\nu \in \NN$ such that $\exp_n(x^\nu) \ge m \ge \exp_n(x^{1/\nu})$ or $\exp_n(x^\nu) \ge 1/m \ge \exp_n(x^{1/\nu})$;
\item[(E2)$_n$] $\M_n$ is a multiplicative subgroup of $\H$ and, for $n \ge 0$, we have $\M_{n-1} \subseteq \M_n$ and for $m \in \M_n$, there are unique $I = I(m) \subseteq \{0,\dots, n\}$ and $m_i = m_i(m) \in \P\M_i$, for $i \in I$, such that $m = \prod_{i \in I} m_i$ (where we make the convention that $\prod_{i \in \emptyset} m_i = 1$), and we have $m \asymp 1$ if and only if $m=1$;
\item[(E3)$_n$] $\E_n$ is an $\RR$-subalgebra of $\H$ and, if $n \ge 0$, then $\E_{n-1} \subseteq \E_n$ and for nonzero $f \in \E_n$, there exist a unique countable ordinal $\alpha$, unique monomials $m_\beta \in \M_n$ and unique nonzero $r_\beta \in \RR$, for $\beta < \alpha$, such that $m_\beta \prec m_\gamma$ for $\gamma < \beta < \alpha$ and the sum $\sum_{\beta < \alpha} r_\beta m_\beta(x)$ converges absolutely to $f(x)$ on some interval $(a,+\infty)$, and such that $f \asymp m_0$;
\item[(E4)$_n$] $\P\E_n^\infty \cup \{0\}$ and $\P\E_n \cup \{0\}$ are $\RR$-vector subspaces of $\H$, $\P\E_n^\infty$ is closed under multiplication, the sets $\P\E_{-\infty}, \P\E_0, \dots, \P\E_n$ are pairwise disjoint and, for large $f \in \P\E_n$ and $n \ge 0$, there exists a nonzero $\nu \in \NN$ such that $\exp_n(x^\nu) \ge |f| \ge \exp_n(x^{1/\nu})$.
\end{itemize}
In the situation of (E3)$_n$ above, we call the monomials $m_\beta$ the \textbf{principal monomials} of $f$, write $$\M(f):= \set{m_\beta:\ \beta < \alpha}$$ and let $$\lm(f):= m_0$$ be the \textbf{leading principal monomial} of $f$.

For the definition of these sets, we distinguish three cases: \medskip

\noindent\textbf{Case $n=-\infty$:} We set $\P\M_{-\infty} = \M_{-\infty} := \{1\}$, $\P\E_{-\infty} = \E_{-\infty} := \RR$ and $\P\E_{-\infty}^\infty := \emptyset$.  We leave the (easy) verification of (E1)$_{-\infty}$--(E4)$_{-\infty}$ to the reader.
\medskip

\noindent\textbf{Case $n=0$:} We define $$\P\M_0 := \set{x^k:\ k \in \ZZ \text{ nonzero}},$$ where $x$ denotes the identity function on $\RR$, and we set $\M_0 := \P\M_0 \cup \{1\}$.  Then (E1)$_0$ and (E2)$_0$ follow immediately, and we define
\begin{equation*}
\E_0 := \set{p\left(\frac1x\right):\ p \text{ is a convergent Laurent series}};
\end{equation*}
then (E3)$_0$ also follows immediately, so we set
\begin{equation*}
\P\E_0 := \set{f \in \E_0:\ 0 \in I(m)  \text{ for every } m \in \M(f)},
\end{equation*}
i.e., $\P\E_0$ consists of those elements of $\E_0$ with zero constant coefficient, and
\begin{equation*}
\P\E_0^\infty:= \set{f \in \P\E_0:\ \text{every } m \in \M(f) \text{ is large}},
\end{equation*}
i.e., $\P\E_0^\infty$ consists of all polynomials in $x$ with zero constant coefficient.
(E4)$_0$ is now straightforward as well.\medskip

\noindent\textbf{Case $n>0$:}  Assume that $\P\M_i$, $\M_i$, $\E_i$ and $\P\E_i$ have been defined with the required properties, for $i = -\infty, 0, \dots, n-1$.  First, we define $$\P\M_n:= \exp \circ \P\E_{n-1}^\infty.$$

\begin{proof}[Proof of {\rm (E1)$_n$}]
(E1)$_n$ follows from (E4)$_{n-1}$: we only prove the last assertion.  Let $m \in \P\M_n$, and let $f \in \P\E_{n-1}^\infty$ be such that $m = \exp \circ f$.  By (E4)$_{n-1}$, there exists a nonzero $\nu \in \NN$ such that $\exp_{n-1}(x^\nu) \ge |f| \ge \exp_{n-1}(x^{1/\nu})$.  If $f(\infty) = +\infty$, then it follows that $\exp_n(x^\nu) \ge m \ge \exp_n(x^{1/\nu})$; if $f(\infty) = -\infty$, then $\exp_n(x^\nu) \ge 1/m \ge \exp_n(x^{1/\nu})$.
\end{proof}

Second, we set
\begin{equation*}
\M_n := \set{\prod_{i \in I} m_i:\ I \subseteq \{0, \dots, n\} \text{ and } m_i \in \P\M_i}.
\end{equation*}

\begin{proof}[Proof of {\rm (E2)$_n$}]
We first prove the uniqueness of $I(m)$ and of $m_i(m)$ for $i \in I$.  Let $m_i \in \P\M_i \cup \{1\}$, for $i=0, \dots, n$ be such that $m:= \prod_{i=0}^n m_i = 1$.  Since $\M_n$ and each $\P\M_i \cup \{1\}$ are multiplicative subgroups of $\H$, it suffices to show that $m_i = 1$ for each $i$.  Suppose, for a contradiction, that $m_i \ne 1$ for some $i$, and fix the largest such $i \in \{0, \dots, n\}$.  Since $m_i = m$ if $i=0$, we must have $i>0$.  By (E1)$_{i-1}$ and (E1)$_i$, there exists a nonzero $\nu \in \NN$ such that $m':= \prod_{j=0}^{i-1} m_j$ satisfies $\exp_{i-1}(x^\nu) \ge \max\{m',1/m'\} \ge \exp_{i-1}(x^{1/\nu})$, while $\exp_i(x^\nu) \ge \max\{m_i,1/m_i\} \ge \exp_i(x^{1/\nu})$.  It follows that $m \ne 1$, as desired.  A similar argument also shows that $m=1$ if and only if $m \asymp 1$.
\end{proof}

Third, we define
\begin{multline*}
\E_n:= \big\{p(m_1, \dots, m_k, M_1, \dots, M_l):\ p \in \Pc{R}{X_1, \dots, X_k}[T_1, \dots, T_l],\\ m_1, \dots, m_k \in \M_n \text{ are small and } M_1, \dots, M_l \in \M_n \text{ are large}\big\},
\end{multline*}
where $\Pc{R}{X_1, \dots, X_k}$ denotes the set of convergent power series.  Note that this definition produces the $\E_0$ defined above in the case $n=0$.

\begin{nrmk}
	\label{support_rmk}
	Let $f = p(m_1, \dots, m_k,M_1, \dots, M_l) \in \E_n$.  Since $p$ is polynomial in $T_1, \dots, T_l$, the infinite sum $p(m_1, \dots, m_k,M_1, \dots, M_l)$ converges absolutely on some interval $(a,+\infty)$ (with $a$ depending on $f$).  Writing $p = \sum_{r \in \NN^k, s \in \NN^l} p_{r,s} X^r T^s$ with $X = (X_1, \dots, X_k)$ and $T = (T_1, \dots, T_l)$, and writing $m = (m_1, \dots, m_k)$ and $M = (M_1, \dots, M_l)$, note that, for $u \in \NN^k$ and $v \in \NN^l$, the set
	\begin{equation*}
	A_{u,v}(m,M):= \set{(r,s) \in \NN^{k+l}:\ m^r M^s = m^u M^v}
	\end{equation*}
	is finite and $\sum_{(r,s) \in A_{u,v}(m,M)} p_{r,s} m^r M^s = p_A m_A$, where $A = A_{u,v}(m,M)$, $m_A = m^uM^v \in \M_n$ and $p_A \in \RR$ only depend on $A$.  On the other hand, the collection $\A = \A(m,M) := \{A_{u,v}(m,M):\ (u,v) \in \NN^{k+l}\}$ is a partition of $\NN^{k+l}$, so the sum $\sum_{A \in \A} p_A m_A$ is absolutely convergent on $(a,+\infty)$.  So we set
	\begin{equation*}
	S(m,M):= \bigcup\set{A \in \A(m,M):\ p_A \ne 0}
	\end{equation*}
	and
	\begin{equation*}
	q:= \sum_{(r,s) \in S(m,M)} p_{r,s} X^r T^s.
	\end{equation*}
	Since the support of $q$ is contained in the support of $p$, the series $q$ belongs to $\Pc{R}{X}[T]$, and we have $f = q(m,M)$.  Replacing $p$ by $q$ if necessary, we may therefore assume that $p_A \ne 0$ if and only if $A \cap \supp(p) \ne \emptyset$, for all $A \in \A$.  In particular, for $\eta \in \M_n$, we may assume that $\eta \in \M(f)$ if and only if $\eta = m^uM^v$ for some $(u,v) \in \supp(p)$.
\end{nrmk}

\begin{proof}[Proof of {\rm (E3)$_n$}]
	Let $f = p(m_1, \dots, m_k,M_1, \dots, M_l) \in \E_n$, and adopt the corresponding notations from Remark \ref{support_rmk}.  Since $m_A \ne m_B$ for any two distinct $A,B \in \A$, we conclude from (E2)$_n$ that $m_A \not\asymp m_B$ for any two distinct $A,B \in \A$.  In other words, the relation $\prec$ linearly orders the set $\{m_A:\ A \in \A \text{ and } p_A \ne 0\}$.  Since $p$ is polynomial in $T$ and $m_1, \dots, m_k$ are small, this ordering is a reverse well-ordering, so we let $\alpha$ be its reverse-order type and rewrite $\sum_{A \in \A} p_A m_A$ as $\sum_{\beta < \alpha} p_\beta m_\beta$. 
	
	Finally, we claim that $f \asymp m_0$: to see this, it suffices to show that $\sum_{0 < \beta < \alpha} p_\beta m_\beta = f - p_0m_0 \prec m_0$.   Note first that $$\frac f{m_0} = \sum_{r \in \NN^k, s \in \NN^l} p_{r,s} m^r \frac{M^s}{m_0};$$ since we can assume, by Remark \ref{support_rmk}, that $m_0$ is the maximal monomial in the support of $p$, it follows that each $m_r \frac{M^s}{m_0}$ is bounded, so that the convergence to $\frac {f(x)}{m_0(x)}$ is not only absolute, but also uniform in $x$ for sufficiently large $x$.  Therefore, $\sum_{0 < \beta < \alpha} p_\beta \frac{m_\beta(x)}{m_0(x)}$ converges absolutely and uniformly to $\frac{f(x) - p_0m_0(x)}{m_0(x)}$.  So let $\epsilon > 0$; absolute and uniform convergence means that there exists a finite $I \subseteq \alpha \setminus \{0\}$ such that $$\sum_{0 < \beta < \alpha,\ \beta \notin I} |p_\beta| \frac{|m_\beta|}{|m_0|} < \frac{\epsilon}2$$ as germs at $+\infty$.  On the other hand, since $I$ is finite and each $m_\beta/m_0$, for $\beta > 0$, is small, we have $$\sum_{\beta \in I} |p_\beta| \frac{|m_\beta|}{|m_0|} < \frac {\epsilon}2.$$  Hence $\left|\frac {f - p_0m_0}{m_0}\right| < \epsilon$, and since $\epsilon > 0$ was arbitrary, the claim and (E3)$_n$ follow.
\end{proof}

Fourth, we set
\begin{equation*}
\P\E_n := \set{f \in \E_n:\ n \in I(m) \text{ for every } m \in \M(f)}
\end{equation*}
and
\begin{equation*}
\P\E_n^\infty:= \set{f \in \P\E_n:\ \text{every } m \in \M(f) \text{ is large}}.
\end{equation*}

\begin{proof}[Proof of {\rm (E4)$_n$}]
	We verify that $\P\E_n^\infty$ is closed under multiplication, leaving the other closure properties for the reader to check.  Given $f = p(m,M)$ and $g = q(m,M)$ in $\P\E_n^\infty$ we may assume, by Remark \ref{support_rmk}, that every monomial $m^\alpha M^\beta$, with $(\alpha,\beta) \in  \supp(p) \cup \supp(q)$, is large.  In addition, we have $n \in I(m^\alpha M^\beta)$ for these monomials, so that $$m^\alpha M^\beta = m'_{\alpha,\beta} \cdot n_{\alpha,\beta}$$ with $m'_{\alpha,\beta} \in \M_{n-1}$ and $n_{\alpha,\beta} \in \P\M_n$ large, for $(\alpha,\beta) \in \supp(p) \cup \supp(q)$.  Therefore, the product of any two such monomials is of the same nature, that is, every principal monomial $m$ of $fg$ is of the form $m = m' n$ with $m' \in \M_{n-1}$ and $n \in \P\M_n$ large, so that $fg \in \P\E_n^{\infty}$.
	
	Finally, if $f \in \P\E_n$ is large, the estimate for $|f|$ follows from the observation in (E3)$_n$ that $f \asymp \lm(f)$ and the estimates in (E1)$_n$.
\end{proof}

This finishes the inductive construction of the sets $\E_n$.  We now set $\E:= \bigcup \E_n$, $\M:= \bigcup_n \M_n$ and define $\eh:\E \into \NN \cup \{-\infty\}$ by $$\eh(f):= \min\set{k:\ f \in \E_k}.$$
We leave it to the reader to check that, for $f \in \E$, we have
\begin{equation*}
\eh(f)= \max\set{\eh(m):\ m \in \M(f)}. 
\end{equation*}

\begin{nrmks}
\label{simple_sum}
\begin{enumerate}
\item The estimates in (E4)$_n$ imply that every $f \in \P\E_n^\infty$ has \textbf{level} $n$, as defined by Marker and Miller in \cite{Marker:1997kn}; in particular, $\level(m) = \eh(m)$ for $m \in \M$.  Also, for $f \in \E$, since $f \asymp \lm(f)$, the levels of $f$ and of $\lm(f)$ are the same if $f$ is infinitely increasing.  Thus, for infinitely increasing $f \in \E$, we have 
\begin{equation}
	\label{level-eh} \level(f) = \level(\lm(f)) = \eh(\lm((f)) \le \eh(f).
\end{equation}
\item Let $l \ge 1$ and $m \in \M_l \setminus \M_{l-1}$, and let $k < l$.  If $n \in \M_k$, then $mn \in \M_l \setminus \M_{l-1}$; that is, $l \in I(mn)$.  Thus, if $f \in \E_k$, then $mf \in \P\E_l$.  It follows that:
\begin{renumerate}
	\item if $g \in \P\E_l$ and $f \in \E_k$, then $fg \in \P\E_l$;
	\item if $g \in \E_l \setminus \E_{l-1}$ and $f \in \E_k$, then $fg \in \E_l \setminus \E_{l-1}$.
\end{renumerate} 
\item Let $f \in \E_n$ be nonzero.  The proof of (E3)$_n$ shows that there are a nonzero $a \in \RR$ and a small $\epsilon \in \E_n$ such that $$f = a \lm(f) (1-\epsilon).$$  To see this, take $a:= p_0$ and $\epsilon:= -\frac{f-p_0m_0}{p_0m_0}$ in the notation used there; this $\epsilon$ belongs to $\E_n$, because $\E_n$ is an $\RR$-algebra containing $\M_n$.
\end{enumerate}
\end{nrmks}

\begin{lemma}
\label{f_going_to_0}
Let $f \in \E_n$ be small.  
\begin{enumerate}
\item There exist $k \in \NN$, $G \in \Pc{R}{X_0, \dots, X_k}$ and small $m_0, \dots, m_k \in \M_n$ such that $G(0) = 0$ and $f = G(m_1, \dots, m_k)$.
\item Let $T$ be a single indeterminate and $P \in \Pc{R}{T}$.  Then $P \circ f \in \E_n$.
\end{enumerate}
\end{lemma}

\begin{proof}
	(1) If $n = 0$, the proof is straightforward, so we assume that $n > 0$.  Let $\kappa,\lambda \in \NN$, small $m_1, \dots, m_\kappa \in \M_n$, large $M_1, \dots, M_\lambda \in \M_n$ and $p \in \Pc{R}{X}[T]$ be such that $f = p(m, M)$, where $X = (X_1, \dots, X_\kappa)$, $T = (T_1, \dots, T_\lambda)$, $m = (m_1, \dots, m_\kappa)$ and $M = (M_1, \dots, M_\lambda)$.  Changing $p$ if necessary, we may assume that $M_1$, \dots, $M_\lambda$ are infinitely increasing.  By Remark \ref{support_rmk}, we may assume that $m^uM^v \in \M(f)$, for every $(u,v) \in \supp(p)$.  Since $f$ is small, every monomial in $\M(f)$ is small, so $m^uM^v$ is small for every $(u,v) \in \supp(p)$; in particular, $(0,0) \notin \supp(p)$.
	
	On the other hand, since $p$ is polynomial in $T$, there is a finite set $B \subseteq \supp(p)$ and, for each $(u,v) \in B$, there is a unit $p_{(u,v)} \in \Pc{R}{X}$, such that $$p(X,T) = \sum_{(u,v) \in B} X^uT^v p_{(u,v)}(X).$$  For $i = 1, \dots, |B|$, we let $X_{\kappa+i}$ be a new indeterminate, and we fix a bijection $\iota:\{1, \dots, |B|\} \into B$ and put $k:= \kappa + |B|$.  Then the series
	\begin{equation*}
	q(X_1, \dots, X_k) := \sum_{i=1}^{|B|} X_{\kappa+i} p_{\iota(i)}(X_1, \dots, X_\kappa)
	\end{equation*}
	belongs to $\Pc{R}{X_1, \dots, X_k}$ and satisfies $q(0) = 0$, the monomials $m_{\kappa+i}:= (mM)^{\iota(i)}$ are small for each $i$, and we have $f = q(m_1, \dots, m_k)$, as required.
	
	Since the collection of convergent power series is closed under composition, part (2) follows from part (1).
\end{proof}

\begin{lemma}
	\label{F-derivative}
	Let $n \ge 0$ and $f \in \E_n$.  Then $f' \in \E_n$.
\end{lemma}

\begin{proof}
	By induction on $n$; the case $n=0$ follows from the fact that the set of convergent Laurent series is  closed under differentiation.  So we assume $n>0$ and the lemma holds for lower values of $n$.  If $f \in \P\M_n$, then $f = \exp(g)$ for some $g \in \P\E_{n-1}^\infty$, so that $$f' = \exp(g) \cdot g' = f \cdot g';$$ in this case, $f' \in \E_n$ by the inductive hypothesis and (E3)$_n$.  From this, the inductive hypothesis and the product rule, the lemma follows easily if $f \in \M_n$ and hence, by the chain rule and the definition of $\E_n$, for general $f$.
\end{proof}

\begin{cor}
	\label{F_isa_field}
	Each $\E_n$, and hence $\E$, is a differential subfield of $\H$.
\end{cor}

\begin{proof}
	Let $f \in \E_n$ be nonzero; by Lemma \ref{F-derivative}, it suffices to show that $1/f \in \E_n$.  By Remark \ref{simple_sum}(3), there are a nonzero $a \in \RR$ and a small $\epsilon \in \E_n$ such that $f = a\lm(f)(1-\epsilon)$.  It follows from Lemma \ref{f_going_to_0} that $G \circ \epsilon \in \E_n$, where $G(T)$ is the geometric series; hence $1/f = \frac{G \circ \epsilon}{a\lm(f)}$ belongs to $\E_n$, as claimed.
\end{proof}

\begin{lemma}
\label{1.5}
Let $f \in \E$.  Then there are unique $f_j \in \P\E_j$, for $j=-\infty,0, \dots, \eh(f)$, such that $ f_{\eh(f)} \ne 0$ and $$f = f_{-\infty} + f_0 + \cdots + f_{\eh(f)}.$$  Moreover, if $f$ is infinitely increasing, then $$\level(f) = \max\set{j= -\infty, 0, \dots, \eh(f):\ f_j \in \I}.$$
\end{lemma}

\begin{proof}
If $f \in \E_0$, we let $f_{-\infty}$ be the constant term of $f$ and set $f_0:= f-f_{-\infty}$.  So we assume that $f \in \E_n$ with $n > 0$.  Let $k,l \in \NN$, small $m_1, \dots, m_k \in \M_n$, large $M_1, \dots, M_l \in \M_n$ and $p \in \Pc{R}{X}[T]$ be such that $f = p(m, M)$, where $X = (X_1, \dots, X_k)$, $T = (T_1, \dots, T_l)$, $m = (m_1, \dots, m_k)$ and $M = (M_1, \dots, M_l)$.  By Remark \ref{support_rmk}, we may assume that $m^uM^v \in \M(f)$, for every $(u,v) \in \supp(p)$.  Thus, we have $\eh(m^rM^s) \le \eh(f)$ for $(r,s) \in \supp(p)$.  So for $\nu = -\infty,0, \dots, \eh(f)$, we set
\begin{equation*}
S_\nu:= \set{(r,s) \in \supp(p):\ \eh(m^rM^s) = \nu}
\end{equation*}
and
\begin{equation*}
p_\nu:= \sum_{(r,s) \in S_\nu} p_{r,s} X^r T^s.
\end{equation*}
Then each $p_\nu$ belongs to $\Pc{R}{X}[T]$, $p_{\eh(f)} \ne 0$, and we have $f = f_{-1} + \cdots + f_{\eh(f)}$, where $f_\nu:= p_\nu(m,M)$ belongs to $\P\E_\nu$ for each $\nu$.

Finally, if $f$ is infinitely increasing, then at least one of the $f_j$ is also infinitely increasing, so it follows that $\lm(f) = \lm(f_{j_0})$ with $j_0 = \max\{j:\ f_j \in \I\}$.
\end{proof}

For $n \ge 0$, we also consider the vector space $$\P\E_n^0:= \set{f \in \P\E_n:\ \text{every } m \in \M(f) \text{ is small}}.$$  Arguing as in the proof of (E4)$_n$, we see that $\P\E_n^0$ is an $\RR$-algebra and, as $\RR$-vector spaces, we have $\P\E_n = \left(\P\E_n^\infty \cup \{0\}\right) \oplus \P\E_n^0$.  Lemma \ref{1.5} now gives the following:

\begin{cor}
\label{direct_sum}
			We have $$\E = \bigoplus \P\E_n = \P\E_{-\infty} \oplus \bigoplus_{n \in \NN} \left(\left(\P\E_n^\infty \cup \{0\}\right) \oplus \P\E_n^0\right)$$ as $\RR$-vector spaces. \qed
\end{cor}

\subsection*{Behaviour under composition}
It follows from part (5) of the first proposition in \cite{Marker:1997kn} and Remark \ref{simple_sum}(2) that for any $f,g \in \I$  we have 
\begin{equation}\label{level_eq}
	\level(f \circ g) = \level(f) + \level(g).
\end{equation}
A corresponding formula does not hold for the exponential height in general: take $g(x) = x+\exp(-x)$ and $f(x) = \exp(x)$.  Then $\eh(f \circ g) = 1 \ne 2 = \eh(f) + \eh(g)$.  This example is a special case of the situation described in Proposition \ref{exp_comp} below.  
Nevertheless, there is a conditional summation formula for the exponential height, see Corollary \ref{eh-lemma} below.

\begin{prop}
	\label{exp_comp}
	Let $f \in \E$ be large.  Then $\exp \circ f = m g \in \E$, where $m \in \M_{\level(f)+1} \setminus \M_{\level(f)}$ and $g \in \E_{\eh(f)}$ is bounded, and we have
	\begin{equation*}
		\eh(\exp \circ f) = \begin{cases}
		\eh(f) &\text{if } \level(f) < \eh(f), \\ \eh(f)+1 &\text{if } \level(f) = \eh(f).
		\end{cases}
	\end{equation*}
	Moreover, if $\eh(f) = \level(f)$, then $\exp \circ f \in \P\E_{\eh(f)+1}$.
\end{prop}

\begin{proof}
	Let $f \in \E$ be large, and let $f_{-\infty} \in \RR$ and $f_i^0 \in \P\E_i^0$ and $f_i^\infty \in \P\E_i^\infty \cup \{0\}$, for $i= 0, \dots, \eh(f)$, be such that $$f = f_{-\infty} + f_0^\infty + f_0^0 + \cdots + f_{\eh(f)}^\infty + f_{\eh(f)}^0,$$ so that $$\exp \circ f = (\exp \circ f_{-\infty}) \cdot \prod_{i=0}^{\eh(f)} \left(\exp \circ f_i^\infty\right) \cdot \left(\exp \circ f_i^0\right).$$
	
	We define $$m:= \exp \circ \left(f_0^\infty + \cdots + f_{\eh(f)}^\infty\right)$$ and $$g:= \exp\circ \left(f_{-\infty} + f_0^0 + \cdots + f_{\eh(f)}^0\right).$$ Since $f$ is large, we have $f_i^\infty = 0$ for $i > \level(f)$ while $f_{\level(f)}^\infty \ne 0$.  It follows that $m \in \M_{\level(f)+1} \setminus \M_{\level(f)}$ by definition of $\M_i$.  On the other hand, $g = E \circ \left(f_0^0 + \cdots + f_{\eh(f)}^0\right) \in \E_{\eh(f)}$ by Lemma \ref{f_going_to_0}(2), where $E$ is the Taylor expansion of $\exp$ at $f_{-\infty}$.  This proves the first statement.
	
	Assume now that $\level(f) = \eh(f)$.  Then $m \in \M_{\eh(f)+1} \setminus \M_{\eh(f)}$, while $\eh(g) \le \eh(f)$.  It follows from Remark \ref{simple_sum}(2a) that $\eh(\exp \circ f) \in \P\E_{\eh(f)+1}$ in this case.  
	
	So we assume for the rest of the proof that $\level(f) < \eh(f)$, so that $f_{\eh(f)}^0 \ne 0$ and $\eh(\exp \circ f) = \eh(mg) \le \eh(f)$.  If $\level(f) = \eh(f)-1$, then $\eh(\exp \circ f) \ge \level(\exp \circ f) = \eh(f)$ by inquality \eqref{level-eh} and equation \eqref{level_eq}, so that $\eh(\exp \circ f) = \eh(f)$ in this subcase.  So we also assume from now on that $\level(f) < \eh(f) - 1$, and we further factorize $g = g_1 g_2$, where $$g_1:= \exp \circ \left(f_{-\infty} + f_0^0 + \cdots + f_{\eh(f)-1}^0\right) \text{ and } g_2:= \exp \circ f_{\eh(f)}^0.$$ Note that $\eh(g_1) < \eh(f)$ by Lemma \ref{f_going_to_0}(2), while $f_{\eh(f)}^0 \ne 0$ implies that $\eh(g_2) = \eh(f)$.  But $\eh(m) < \eh(f)$ in this last subcase, so $\eh(\exp \circ f) = \eh(g) = \eh(g_2) = \eh(f)$ by Remark \ref{simple_sum}(2a).
\end{proof}

\begin{lemma}
\label{0-case}
Let $n \in \NN$ and $g \in \P\E_n^\infty$ be infinitely increasing.
\begin{enumerate}
\item For nonzero $k \in \ZZ$, we have $g^k \in \P\E_n^\infty$ if $k > 0$, and $g^k \in \P\E_n^0$ if $k < 0$.
\item If $f \in \P\E_0^\infty$, then $f \circ g \in \P\E_n^\infty$.
\item If $f \in \P\E_0$, then $f \circ g \in \P\E_n$.
\end{enumerate}
\end{lemma}

\begin{proof}
	(1) Since $\P\E_n^\infty$ and $\P\E_n^0$ are closed under multiplication, it suffices to show that $1/g \in \P\E_n^0$ for $g \in \P\E_n^\infty$.  Let $g \in \P\E_n^\infty$; by Remark \ref{simple_sum}(3), we can write $g = a\lm(g)(1 - \epsilon)$, with $a \in \RR$ nonzero and $\epsilon \in \E_n$ small; in particular, every principal monomial of $\epsilon$ is small.  Hence $1/g = \frac{G\circ h}{r\lm(g)}$, where $G$ is the geometric series.  Since $n \in I(\lm(g))$ and $\lm(g)$ is large, and since every principal monomial of $\epsilon$ is small, it follows that $1/g \in \P\E_n^0$.
	Parts (2) and (3) follow from part (1).
\end{proof}

\begin{lemma}
\label{composition_lemma}
Let $p,n \in \NN$ with $p \ge 1$, $f \in \P\M_p$  and $g \in \P\E_n^\infty$ be infinitely increasing.  Then $f \circ g \in \P\M_{p+n}$.
\end{lemma}

\begin{proof}
	Let $f_0 \in \P\E_{p-1}^\infty$ be such that $f = \exp \circ f_0$; we proceed by induction on $p$.  If $p=1$ then, by Lemma \ref{0-case}(2), $f_0 \circ g \in \P\E_n^\infty$, so $f \circ g \in \P\M_{1+n}$ by definition, as required.  So we assume $p > 1$ and the lemma holds for lower values of $p$.  By the inductive hypothesis and (E3)$_{p-1}$, we have $m \in \M(f_0)$ if and only if $m \circ g \in \M(f_0 \circ g)$. It follows, also from the inductive hypothesis, that $f_0 \circ g \in \P\E_{p-1+n}^\infty$.  Hence $f \circ g \in \P\M_{p+n}$ by definition.
\end{proof}

We set $\P\E^\infty := \bigcup_{n \in \NN} \P\E_n^\infty$.

\begin{cor}
	\label{eh-lemma}
	\begin{enumerate}
		\item Let $f \in \E$ and $g \in \P\E^\infty$ be infinitely increasing.  Then $\eh(f \circ g) = \eh(f) + \eh(g)$, and if $f \in \P\E_{\eh(f)}$, then $f \circ g \in \P\E_{\eh(f) + \eh(g)}$.
		\item $\E_n \subseteq \E_{n+1} \circ \log$ for $n \in \NN$.
	\end{enumerate}
\end{cor}

\begin{proof}
	Part (1) is trivial if $\eh(f) = -\infty$, follows from Lemma \ref{0-case}(3) if $\eh(f) = 0$ and follows from the definition of $\eh(f)$ and Lemma \ref{composition_lemma} if $\eh(f) > 0$.  
	For part (2), given $f \in \E_n$, we have $f \circ \exp \in \E_{n+1}$ by part (1), so that $f = f \circ \exp \circ \log \in \E_{n+1} \circ \log$.
\end{proof}

\begin{prop}
\label{1.9}
Let $f \in \H$.  Then $f \in \E$ if and only if $f$ is definable by some $\Lanexp$-term.  
\end{prop}

\begin{proof}
It is clear from the definition of $\E$ that every $f \in \E$ is definable by some $\Lanexp$-term.  For the converse, since $\E$ is an $\RR$-subalgebra of $\H$ and $x, \exp \in \E$, it suffices to show the following:
\begin{renumerate}
	\item if $p \in \Pc{R}{X_1, \dots, X_k}$ and $f_1, \dots, f_k \in \E$ are such that the point $(f_1(\infty), \dots, f_k(\infty))$ lies in the domain of convergence of $p$, then $p(f_1, \dots, f_k) \in \E$;
	\item if $f \in \E$, then $\exp \circ f \in \E$.
\end{renumerate}

For (i), after replacing $p$ by its convergent Taylor series at the point $(f_1(\infty), \dots, f_k(\infty))$ if necessary, we may assume that $f_1, \dots, f_k$ are small.  The claim now follows from Lemma \ref{f_going_to_0}(2).  

For (ii), we may assume by (i) that $f$ is large; so (ii) follows from Proposition \ref{exp_comp}.
\end{proof}

\subsection*{Closing under $\log$}  We obtain $\H$ from $\E$ by composing with $\log$ on the right:

\begin{lemma}
\label{log_push}
Let $f \in \E_n$ be such that $f>0$. Then $\log \circ f \in \E_{n+1} \circ \log$.
\end{lemma}

\begin{proof}
The conclusion is trivial if $n = -\infty$, so we assume $n \ge 0$.  By Remark \ref{simple_sum}(2), there exist $m \in \M_{n}$ and $g \in \E_{n}$ such that $g \asymp 1$ and $f = mg$.  So $\log \circ f = \log \circ m + \log \circ g$, and $\log \circ g$ belongs to $\E_n \subseteq \E_{n+1} \circ \log$ by Corollary \ref{eh-lemma}(2).  On the other hand, if $n = 0$, then $m = x^k$ for some nonzero $k \in \ZZ$, so $\log \circ m \in \E_0 \circ \log$.  If $n>0$, then by definition $m = \exp \circ h$ for some $h \in \E_{n-1}$, so $\log \circ m = h \in \E_{n-1} \subseteq \E_{n+1} \circ \log$, which finishes the proof of the lemma.
\end{proof}

\begin{prop}
\label{1.11}
We have $\H = \bigcup_{k \in \NN} \E \circ \log_k$; in particular, for large $f \in \H$, there exist $k,l,\nu \in \NN$ such that $\level(f) = k-l$ and $$\exp_k(x^\nu) \circ \log_l \ge |f| \ge \exp_k(x^{1/\nu}) \circ \log_l.$$
\end{prop}

\begin{proof}
Since $\E \subseteq \E\circ\log \subseteq \E\circ\log_2 \subseteq \cdots$, it follows from Lemma \ref{log_push} and Proposition \ref{1.9} that $\bigcup_{k \in \NN} \E \circ \log_k$ contains all functions of $\H$ that are defined by $\Lanexplog$-terms.  Equality between the two sets then follows from \cite[Corollary 4.7]{Dries:1994tw}.  The estimates follow from (E4).
\end{proof}

In view of the previous proposition, for $h \in \H$ and $k \in \NN$ and $f \in \E$ such that $h = f \circ \log_k$, we let $$\M(h):= \set{ m \circ \log_k:\ m \in \M(g)}$$ be the set of \textbf{principal monomials} of $h$ and $$\lm(h):= \lm(g) \circ \log_k$$ be the \textbf{leading monomial} of $h$.

\begin{lemma}
\label{log_comp_lemma}
	Let $g_1,g_2 \in \E$ and $k_1, k_2 \in \NN$ be such that $k_1 \le k_2$ and $g_1 \circ \log_{k_1} = g_2 \circ \log_{k_2}$.  Then $\eh(g_2) = \eh(g_1) + k_2-k_1$.
\end{lemma}

\begin{proof}
	The hypothesis implies that $g_1 \circ \exp_{k_2-k_1} = g_2$, so the conclusion follows from Corollary \ref{eh-lemma}(1).
\end{proof}

Justified by Lemma \ref{log_comp_lemma}, we extend $\eh$ uniquely to all of $\H$ as follows: given $f \in \H$, choose $g \in \E$ and $k \in \NN$ such that $f = g \circ \log_k$, and put $$\eh(f):= \eh(g)-k.$$

\begin{cor}
\label{eh-cor}
\begin{enumerate}
\item Let $f \in \H$.  Then $\eh(f \circ \exp) = \eh(f) + 1$ and $\eh(f \circ \log) = \eh(f) -1$.
\item Let $f \in \H$ be infinitely increasing.  Then $$\level(f) = \eh(\lm(f)) \le \eh(f).$$
\item For $k \in \NN$ the set $$\H_{\le k}:= \set{ f \in \H:\ \eh(f) \le k}$$ is a differential subfield of $\H$.
\end{enumerate}
\end{cor}

\begin{proof}
(1) Let $g \in \E$ and $k \in \NN$ be such that $f = g \circ \log_k$.  Note that $\eh(f \circ \log) = \eh(f)-1$ by definition of $\eh$.  If $k>0$, the same is true for $f \circ \exp$; whereas if $k=0$, then $f \in \E$ and $\eh(f \circ \exp) = \eh(f)+1$ by Corollary \ref{eh-lemma}(1).

Part (2) follows from the definition of $\eh(f)$ and inequality \eqref{level-eh}.

For part (3), let $k \in \NN$ and $f \in \H_{\le k}$.  Let also $i,j \in \NN$ and $g \in \E_i$ be such that $f = g \circ \log_j$.  Then $$f' = (g' \circ \log_j) \frac1{\log_0 \cdots \log_{j-1}} = \frac{g'}{\exp_j \cdots \exp_1} \circ \log_j \in \H_{\le k},$$ so part (3) follows from Corollary \ref{F_isa_field}.
\end{proof}

For each $i \in \ZZ \cup \{-\infty\}$, we set
\begin{equation*}
\H_i := \set{f \in \H:\ \eh(f) = i} \cup \{0\}.
\end{equation*}
By Corollary \ref{eh-lemma}(1), each $\H_i$ is an $\RR$-vector subspace of $\H$, and we have $\H_i \cap \H_j = \{0\}$ for $i \ne j$.

\begin{cor}
\label{direct_sum2}
As $\RR$-vector spaces, we have 
\begin{align*}
\H = \bigoplus_{n \in \ZZ \cup \{-\infty\}} \H_n. &\qed
\end{align*}
\end{cor}
Finally, we set $$\E^\infty := \bigoplus_{n \in \NN} \left(\P\E_n^\infty \cup \{0\}\right)$$ and $$\U:= \bigcup_{k \in \NN} \E^\infty \circ \log_k;$$ the germs in $\U$ are called \textbf{purely infinite}.    This set $\U$ is an $\RR$-vector subspace of $\H$ and, by Corollary \ref{direct_sum} and Proposition \ref{1.11}, we have $$\H = \U \oplus \B,$$  where $\B$ is the $\RR$-vector subspace of all bounded germs of $\H$.  We set $$\la:= \exp \circ\, \U,$$ a multiplicative $\RR$-vector subspace of $\H^{>0}$.

\begin{prop}
	\label{transmonomials}
	\begin{enumerate}
		\item $\la = \bigcup_{k \in \NN} \M \circ \log_k$, that is, $\la$ is the set of principal monomials of $\H$.
		\item We have $h \asymp \lm(h)$, for $h \in \H$; in particular, every Archimedean class of $\H^{>0}$ contains exactly one representative from $\la$.
		\item Every $h \in \H$ is of the form $p(m,M)$, where $p \in \Pc{R}{X}[Y]$, $X = (X_1, \dots, X_k)$, $Y = (Y_1, \dots, Y_l)$ for some $k,l \in \NN$, $m = (m_1, \dots, m_k)$ and $M = (M_1, \dots, M_l)$ with each $m_i \in \la$ small and each $M_j \in \la$ large.
	\end{enumerate}
\end{prop}

\begin{proof}
	(1) By definition, we have 
	\begin{equation}\label{monomial_sets}
	\exp \circ\,\E^\infty = \prod_{n \in \NN} \left(\P\M_{n+1} \cup \{1\}\right) \subseteq \M,
	\end{equation}  
	so that $\la = \bigcup_{k \in \NN} (\exp \circ \E^\infty) \circ \log_k \subseteq \bigcup_{k \in \NN} \M \circ \log_k$.  On the other hand, we have $x = \exp \circ \log \in (\exp \circ \E^\infty) \circ \log$, while Lemma \ref{composition_lemma} implies that $\E^\infty \circ \exp \subseteq \E^\infty$, so that $\exp \circ \E^\infty \subseteq \exp \circ \E^\infty \circ \log$.  Since $\exp \circ \E^\infty \circ \log$ is a multiplicative group, it follows from the equality in \eqref{monomial_sets} that $\M \subseteq \exp \circ \E^\infty \circ \log$.  Therefore, we get $\bigcup_{k \in \NN} \M \circ \log_k \subseteq \bigcup_{k \in \NN} (\exp \circ \E^\infty) \circ \log_k$, which proves part (1).
	
	(2) Let $h \in \H$, and let $g \in \E$ and $k \in \NN$ be such that $h = g \circ \log_k$.  Then $g \asymp \lm(g)$ by $(E3)$, so that $h \asymp \lm(g) \circ \log_k = \lm(h)$.  On the other hand, no two distinct germs in $\M$ are in the same Archimedean class by (E2).  Also, by Lemma \ref{composition_lemma}, we have $\M = \M \circ \exp \circ \log \subseteq \M \circ \log \subseteq \M \circ \log_2 \subseteq \cdots$, so it follows from part (1) that no two distinct germs in $\la$ are in the same Archimedean class.
	
	(3) follows from the definition of $\E$, part (1) and Proposition \ref{1.11}.
\end{proof}

\section{Real domains}
\label{real_domains}

In this section, we introduce real domains and study the $\LL$-analytic continuations of some elementary germs in $\H$.  We let $\Log:\LL \into \CC$ be the biholomorphic map $$\Log x:= \log|x| + i \arg x,$$ and we let $\Exp:\CC \into \LL$ be its inverse.  Corresponding to this chart, we define $$d(x,y):= |\Log x - \Log y|$$ and set $B(x,s):= \set{y \in \LL:\ d(x,y)<s}$ for $x \in \LL$.

Here is a general observation about analytic continuations on $\LL$:

\begin{lemma}
\label{L-extension}
Let $U \subseteq \CC \setminus \{0\}$ be a domain and $f:U \into \CC \setminus\{0\}$ be holomorphic, and assume that $\pi^{-1}(U)$ is simply connected.  Then $f$ lifts to a unique holomorphic $\f:\pi^{-1}(U) \into \LL$ such that $f \circ \pi = \pi \circ \f$.
\end{lemma}

\begin{proof}
Note that $\exp(\Log(\pi^{-1}(U))) = U$ and that $V:= \Log(\pi^{-1}(U))$ is a simply connected domain in $\CC$.  By assumption, both $f \circ \exp:V \into \CC\setminus\{0\}$ and $1/f \circ \exp:V \into \CC\setminus\{0\}$ are holomorphic, so by \cite[Theorem 13.11]{Rudin:1987fk}, there exists a holomorphic $g:V \into \CC$ such that $f \circ \exp = \exp \circ g$.  Now define $\f:= \Exp \circ g \circ \Log$; since $\pi \circ \Exp = \exp$ and $\exp \circ \Log = \pi$, it follows that $f \circ \pi = \pi \circ \f$.
\end{proof}

\begin{expls}
\label{basic_extensions_2}
Particular examples of the previous lemma include:
\begin{enumerate}
\item $f \in \CC(X)$ and $U= D(a)$, where $a>0$ is such that $Z(f):= \set{z \in \CC:\ f(z) = 0} \subseteq B(0,a):= \set{ w \in \CC:\ |w| < a}$ and $$D(a):= \set{z \in \CC:\ |z|>a};$$
\item for $a \in \RR$, the injective map $t_a: D(|a|) \into \CC \setminus \{0\}$ defined by $t_a(z):= z+a$.  Note that, in this case, the corresponding lifting extends to a biholomorphic map $\t_a:\LL \setminus \pi^{-1}(-a) \into \LL \setminus \pi^{-1}(a)$ with compositional inverse $\t_{-a}$;
\item $f \in \Pc{R}{T}$ is nonzero, where $T$ is a single indeterminate, and $a>0$ is sufficiently small such that $f(z) \ne 0$ for $z \in U:= B(0,a) \setminus \{0\}$.  In this situation, if $f(0) = 0$, then $|\f(x)| \to 0$ as $|x| \to 0$; while if $f(0) = a \ne 0$, then $d(\f(x),a) \to 0$ as $|x| \to 0$;
\item The map $\exp:\CC \setminus \{0\} \into \CC \setminus \{0\}$; we denote the corresponding analytic continuation by $\eexp:\LL \into \LL \setminus \{(1,0)\}$.  Note that $\eexp = \Exp \circ \pi$.
\end{enumerate}
\end{expls}

However, we need more precise information on the kinds of simply connected domains on which definable functions have analytic continuations.
We set $$H(a):= \set{z \in \CC:\ \re z > a},$$ for $a \in \RR \cup \{-\infty\}$, and $$S:= \set{z \in \CC:\ |\im z| < \frac{\pi}{2}}.$$  We denote by $\log:H(0) \into S$ the main branch of the logarithm on $H(0)$ and by $\arg:H(0) \into \left(-\frac{\pi}{2}, \frac{\pi}{2}\right)$ the main branch of the argument; note that, for $a>1$, we have $\log H(a) \subseteq H(\log a)$.


\begin{df}
\label{real domain}
A set $U \subseteq H(0)$ is a \textbf{real domain} if there exist $a \ge 0$ and a continuous function $f = f_{U,\arg}:(a,+\infty) \into (0,\frac{\pi}{2}]$ such that $$U \cap D(a) = \set{z \in D(a):\ |\arg z| < f(|z|)}.$$  Note that, in this situation, $U \cap D(a)$ is definable if $f_{U,\arg}$ is definable.
\end{df}

\begin{expls}
\label{real_domain_expl}
The following are definable real domains:
\begin{enumerate}
\item $H(a)$ for $a \ge 0$;
\item for $0 < \alpha \le \frac{\pi}{2}$, the \textbf{sector} $S(\alpha):= \set{z \in H(0):\ |\arg z| < \alpha}$.
\end{enumerate}
\end{expls}

A special class of definable real domains are the following:

\begin{df}
\label{strict_real_domain}
A set $U \subseteq H(0)$ is a \textbf{standard domain} if there exist $a \ge 0$ and a continuous function $f = f_{U,\im}:(a,+\infty) \into (0,+\infty)$ such that $$U \cap H(a) = \set{z \in H(a):\ |\im z| < f(\re z)}.$$  Note that, in this situation, $U \cap H(a)$ is definable if and only if $f_{U,\im}$ is definable.
\end{df}

For example, every sector $S(\alpha)$, for $\alpha < \frac\pi2$, 
is a definable standard domain; but the half-planes $H(a)$, for $a>0$, are not standard domains.  

\begin{lemma}
\label{tame_domain_lemma}
Let $U \subseteq H(0)$.  If $\,U$ is a definable standard domain, then $U$ is a definable real domain.  Moreover, if $V \subseteq U$ is a definable real domain, then $V$ is a definable standard domain.
\end{lemma}

\begin{proof}
Assume there exist $b \ge 0$ and a definable $g = f_{U,\im}:(b,+\infty) \into (0,+\infty)$ such that $$U \cap H(b) = \set{z \in H(b):\ |\im z| < g(\re z)}.$$  The functions $\rho_g,\theta_g:(b,+\infty) \into (0,+\infty)$ defined by $$\rho_g(x):= \sqrt{x^2 + g(x)^2} \quad\text{ and }\quad \theta_g(x):= \arctan(g(x)/x)$$ are definable, and $\rho_g$ is infinitely increasing.  So there exists $c \ge b$ such that $\rho_g\rest{(c,+\infty)}$ is injective with compositional inverse $\rho_g^{-1}:(a,+\infty) \into (c,+\infty)$.  Therefore, we set $$\xi_g:= \theta_g \circ \rho_g^{-1},$$ and we have $U \cap H(a) = \set{z \in H(a):\ |\arg z| < \xi_g(|z|)}$.  Increasing $a$ if necessary, we conclude that $U \cap D(a) = \set{z \in D(a):\ |\arg z| < \xi_g(|z|)}$.

Moreover, let $d \ge a$ and $f:(d,+\infty) \into (0,\pi/2)$ be such that $V := \set{z \in D(a):\ |\arg z| < f(|z|)} \subseteq U$; then $f \le f_{U,\arg} = \xi_g$.  Define $\eta_f:(d,+\infty) \into (0,+\infty)$ by $$\eta_f(x):= x \cdot (\cos \circ f)(x).$$  Note that, for $z \in \partial U$, we have $\re z = |z| \cdot (\cos \circ \xi_g)(|z|)$; since $U$ is a definable standard domain, it follows that the function $x \mapsto x \cdot (\cos \circ \xi_g)(x)$ is infinitely increasing.  Since $\cos\rest{(0,\pi/2)}$ is decreasing and $\eta_f$ is definable, it follows that $\eta_f$ is infinitely increasing as well.  Increasing $d$ if necessary, we may assume that $\eta_f$ is strictly increasing with image $(e,+\infty)$, for some $e>0$, and we denote by $\eta_f^{-1}:(e,+\infty) \into (d,+\infty)$ its compositional inverse.  Then $$V \cap H(e) = \set{z \in H(e):\ |\im z| < \left(\eta_f^{-1}\cdot\left(\sin \circ f \circ \eta_f^{-1}\right)\right)(\re z)},$$ which shows that $V$ is a standard domain.
\end{proof}

It is convenient to talk about \textit{germs} of sets in $H(0)$ at $\infty$: for two subsets $A,B \subseteq H(0)$, we set $A \sim B$ if and only if there exists $a>0$ such that $A \cap D(a) = B \cap D(a)$.  The corresponding equivalence classes of subsets of $H(0)$ are called \textbf{germs} of subsets of $H(0)$ \textbf{at $\infty$}.  We will not explicitely distinguish between subsets of $H(0)$ and their germs at $\infty$ when the meaning is clear from context.  In this sense, every definable real domain $U$ corresponds to a unique element $f = f_{U,\arg} \in (0,\frac\pi2]_\H$, where
$$(0,\pi/2]_\H := \set{h \in \H:\ 0 < h \le \frac\pi2},$$ and we write $U = U_f$ in this sense; the standard domains correspond to $$\Hsr:= \set{f \in (0,\pi/2]_\H:\ f = f_{U,\arg} \text{ for some standard } U}.$$  It follows from Lemma \ref{tame_domain_lemma} that:

\begin{cor}
	\label{Hsr_cut}
	$\Hsr$ is a cut in $(0,\pi/2)_\H$. \qed
\end{cor}

The notion of real domain also makes sense on the Riemann surface $\LL$ of the logarithm: for $a \ge 0$, we set $$H_\LL(a):= \set{(r,\varphi) \in \LL:\ r > a}.$$ 

\begin{df}
\label{real_domain_2}
Let $\frU \subseteq \LL$.
We call $\frU$ a \textbf{real domain} if there exist $a \ge 0$ and a continuous function $f = f_{\frU,\arg}:(a,+\infty) \into (0,+\infty]$ such that $$\frU = \frU_f :=
\set{z \in H_\LL(a):\ |\arg z| < f(|z|)}.$$  Note that $\frU$ is simply connected, and that $\frU$ is definable if and only if $f_{\frU}$ is.
We call $\frU$ \textbf{angle-bounded} if $f_{\frU}$ is bounded.
\end{df}

For example, for $a>0$, the set $H_\LL(a)$ is a definable real domain and, for $\alpha > 0$, the \textbf{strip} or \textbf{sector} $$S_\LL(\alpha) := \set{x \in \LL:\ |\arg x| < \alpha}$$ is a definable and angle-bounded real domain.

We call a real domain $\frU \subseteq S_\LL(\pi/2)$ a \textbf{standard domain} if $\pi(\frU) \subseteq H(0)$ is a standard domain.

\begin{lemma}
\label{image_of_real_domains}
Let $\frU, \frV \subseteq \LL$ be real domains and $\varphi:\frV \into \LL$ be holomorphic and injective.  Assume that $\frU \subseteq \frV$ and $\varphi(\frU)$ and $\varphi(\frV)$ are  real domains.  Then $f_{\frU} < f_{\frV}$ if and only if $f_{\varphi(\frU)} < f_{\varphi(\frV)}$.
\end{lemma}

\begin{proof}
It is straightforward to see that $\frU \subsetneq \frV$ if and only if $f_{\frU} < f_{\frV}$.  
Since $\varphi$ is injective, it follows that $f_{\frU} < f_{\frV}$ iff $\frU \subsetneq \frV$ iff $\varphi(\frU) \subsetneq \varphi(\frV)$ iff $f_{\varphi(\frU)} < f_{\varphi(\frV)}$.
\end{proof}

\begin{expl}
\label{basic_extensions}
Note that $\Log(H_\LL(1)) = H(0)$ and, for $x \in S_\LL(\pi/2)$, we have $\Log(x) = \log(\pi(x))$.  Thus we define $\llog:H_\LL(1) \into S_\LL(\pi/2)$ by $$\llog(x) := (\pi\rest{S(\pi/2)})^{-1}(\Log(x));$$ note that $\eexp\rest{S_\LL(\pi/2)} = \llog^{-1}$.

For real $r>0$, we also define the \textbf{power function} $\p_r:\LL \into \LL$ by $$\p_r(s,\theta):= (s^r,r\theta)$$ and the \textbf{scalar multiplication} $\m_r:\LL \into \LL$ by $$\m_r(s,\theta):= (rs,\theta).$$
Thus, if $\varphi$ is any of $\llog$, $\p_r$ or $\m_r$ we have $\varphi:\LL \into \LL$, while $\eexp:U_{\pi/2} \into \LL$.
\end{expl}

We also work near $\infty$, in the following sense: two sets $\frA,\frB \subseteq \LL$ are called \textbf{equivalent at $\infty$} if there exists $R>0$ such that $\frA \cap H_\LL(R) = Y\frB \cap H_\LL(R)$.  The corresponding equivalence classes are called the \textbf{germs at $\infty$} of subsets of $\LL$.  We will not explicitely distinguish between subsets of $\LL$ and their germs at $\infty$ when the meaning is clear from context.  In this sense, every definable real domain $\frU \subseteq \LL$ corresponds to a unique element $f = f_{\frU} \in \bar{\H^{>0}}:= \H^{>0} \cup \{+\infty\}$, where
$$\H^{>0} := \set{h \in \H:\ h > 0},$$ and we write $\frU = \frU_f$ in this sense.

Given germs $\frU,\frV$ at $\infty$ of domains in $\LL$, we write ``$\varphi:\frU \into \frV$'' to mean that $\varphi$ is a map defined on some domain in $\LL$ whose germ at $\infty$ is $\frU$ and taking values in some domain in $\LL$ whose germ at $\infty$ is $\frV$.

\subsection*{Basic images} 
Our first goal is to understand images of definable real domains under power functions, $\llog$ and $\eexp$.

\begin{expl}
\label{log_of_real}
Let $\frU \subseteq H_\LL(1)$ be a definable real domain.  Then direct computation shows that $\llog(\frU)$ is a definable standard domain.
\end{expl}

\begin{df}
\label{arg_functions}
Let $\frU,\frV$ be germs at $\infty$ of domains in $\LL$ and $\varphi:\frU \into \frV$.
\begin{enumerate}
\item If $\varphi\rest{\frU}$ is injective and $\varphi(\frU')$ is (the germ at $\infty$ of) a real domain, for every real domain $\frU' \subseteq \frU$, we say that $\varphi$ \textbf{maps real domains to real domains} and write $\varphi:\frU \xrightarrow{\text{real}} \frV$.
\item If $\varphi\rest{\frU}$ is injective and $\varphi(\frU')$ is a definable real domain, for every definable real domain $\frU' \subseteq \frU$, we say that $\varphi$ \textbf{maps definable real domains to definable real domains} and write $\varphi:\frU \xrightarrow{\text{dfr}} \frV$.
\item If $h \in \bhplus$ and $\varphi:\frU_h \xrightarrow{\text{dfr}} \frV$, we denote  by $\nu_\varphi:(0,h)_{\H} \into \hplus$ the map defined by $$\nu_\varphi(g):= f_{\varphi(\frU_g)}.$$ If, moreover, we have that $\varphi$ is an analytic continuation of $f \in \H$, we set $\nu_f:= \nu_\varphi$.
\end{enumerate}
\end{df}

\begin{expls}
\label{arg_function_expls}
Let $f \in \hplus$.
\begin{enumerate}
\item For $r>0$, the maps $\m_r$ and $\p_r$ are injective, and direct computation shows that $$\nu_{m_r}(f) = f \circ m_{1/r}$$ and $$\nu_{p_r}(f) = m_r \circ  f \circ p_{1/r}$$ where $m_s:\RR \into \RR$ and $p_s:[0,\infty) \into [0,\infty)$ are defined by $m_s(x):= sx$ and $p_s(x):= x^s$, for $s>0$.  In particular, $\m_r:\LL \xrightarrow{\text{dfr}} \LL$ and $\p_r:\LL \xrightarrow{\text{dfr}} \LL$; moreover, $\p_r:H_{\LL}(1) \xrightarrow{\text{dfr}} H_{\LL}(1)$.
\item We define $\rho_{\log}(f) := \sqrt{\log^2 + f^2}$ and $\theta_{\log}(f):= \arctan(f/\log)$.  Then $\rho_{\log}(f) \in \I$, and $$\nu_{\log}(f) = \theta_{\log}(f) \circ \left(\rho_{\log}(f)\right)^{-1} \in \Hsr.$$  It follows that $\llog:H_{\LL}(1) \xrightarrow{\text{dfr}} S_{\LL}(\pi/2)$ with $\nu_{\log}:\hplus \into \Hsr$.
Since $\eexp\rest{S_\LL(\pi/2)} = \left(\llog\rest{H_\LL(1)}\right)^{-1}$, it follows that $\eexp:S_{\LL}(\pi/2) \xrightarrow{\text{dfr}} H_{\LL}(1)$, that both $\nu_{\log}$ and $\nu_{\exp}:\Hsr \into \H^{>0}$ are injective and $\nu_{\exp}$ is the compositional inverse of $\nu_{\log}$.
\item For $a \in \RR$, an elementary calculation shows that, if $f$ is bounded, then $\t_a:\frU_f \xrightarrow{\text{real}} \frU_f$ and $\t_a:\frU_f \xrightarrow{\text{dfr}} \frU_f$, so that $\nu_{\t_a}:\H^{>0} \setminus \I \into \H^{>0} \setminus \I$; moreover, we have $\nu_{\t_a}^{-1} = \nu_{\t_{-a}}$.  On the other hand, if $f \in \I$, then $\t_a(\frU_f)$ is not a definable real domain---in fact, it is not even a real domain in general.  However, the same calculation as for bounded $f$ shows that, for $f \in \I$ and real $\epsilon > 0$, we have (as germs)
\begin{equation}
\label{translation_image}
\frU_{f\circ t_{-\epsilon}} \subseteq \t_a(\frU_f) \subseteq \frU_{f\circ t_\epsilon}.
\end{equation}
\end{enumerate}
\end{expls}

The next lemma, together with the previous examples, can be used to describe $\nu_w$ for every finite composition $w$ of $\exp$, $\log$ and $p_r$ with $r>0$---at least in principle.  As we do not need the details here, we leave it to the reader to explore this.

\begin{lemma}
	\label{arg_fun_comp}
	Let $g,h \in \bar{\H^{>0}}$, $\frU$ the germ at $\infty$ of some domain in $\LL$, $\varphi:\frU_g \xrightarrow{\text{dfr}} \frU_h$ and $\psi:\frU_h \xrightarrow{\text{dfr}} \frU$.  Then $\psi \circ \varphi:\frU_g \xrightarrow{\text{dfr}} \frU$ and $$\nu_{\psi \circ \varphi} = \nu_\psi \circ \nu_\varphi.$$
\end{lemma}

\begin{proof}
	Straightforward; we leave the details to the reader.
\end{proof}

\section{Angular level}
\label{angular_level}

To describe the analytic continuation properties of germs in $\H$, we need a rough measure of the size of a real domain.  
We first recall the main properties of level.

\begin{facts}[the Proposition in \cite{Marker:1997kn}]
	\label{level_facts}
	Let $f,g \in \I$.
	\begin{enumerate}
		\item $\level(\exp) = 1$, $\level(\log) = -1$ and, for real $r>0$, $\level(p_r) = 0$.
		\item If $f \le g$, then $\level(f) \le \level(g)$.
		\item $\level(f \circ g) = \level(f) + \level(g)$.
		\item $\level(fg) = \level(f+g) = \max\{\level(f),\level(g)\}$.
		\item If $u \in \H$ is such that $0 < u(+\infty) < +\infty$, then $\level(fu) = \level(f)$.
	\end{enumerate}
\end{facts}

We extend the level to all of $\hplus$ as follows: we set
\begin{equation*}
\D := \set{1/f:\ f \in \I},
\end{equation*}
and for $n \in \ZZ$, we let $\I_n$ be the set of all $f \in \I$ of level $n$, and we set
\begin{equation*}
\D_n := \set{1/f:\ f \in \I_n}.
\end{equation*}
For $f \in \D$, we define the \textbf{level} of $f$ to be the unique $n \in \ZZ$ such that $f \in \D_n$.  Furthermore, we set $\Hone:= \H^{>0} \setminus (\I \cup \D)$ and $\level(f):= -\infty$, for $f \in \Hone$.
Note that $\B \cap \H^{>0} = \Hone \cup \D$.

\begin{prop}
	\label{D_level_lemma}
	Let $f,g \in \hplus$.
	\begin{enumerate}
		\item If $f,g \in \I \cup \Hone$ and $f \le g$, then $\level(f) \le \level(g)$.
		\item If $f,g \in \D \cup \Hone$ and $f \le g$, then $\level(f) \ge \level(g)$.
		\item If $g \in \I$, then $\level(f \circ g) = \level(f) + \level(g)$.
		\item $\level(fg) \le \max\{\level(f),\level(g)\}$; equality holds whenever $\level(f) \ne \level(g)$.
		\item If $f,g \in \D$ and $h$ is a nonzero $\RR$-linear combination of $f$ and $g$, then $\level(h) \ge \max\{\level(f), \level(g)\}$.
		\item If $f,g \in \U$ and $h$ is a nonzero $\RR$-linear combination of $f$ and $g$, then $\level(h) \le \max\{\level(f), \level(g)\}$.
		\item If $f \asymp g$, then $\level(f) = \level(g)$.
	\end{enumerate}
\end{prop}

\begin{proof}
	Parts (1)--(3) follow directly from Facts \ref{level_facts}; we leave the details to the reader.  For part (4), we assume that $f \in \I$ and $g \in \D$; all other cases follow easily from Facts \ref{level_facts}.  The desired inequality clearly holds if $\level(fg) = -\infty$, so we also assume that $\level(fg) \in \ZZ$.
	In this situation, if $fg \in \I$, then the equality $f = (fg)/g$ implies \begin{align*}
	\max\{\level(f),\level(g)\} &\ge \level(f) \\ &= \max\{\level(fg), \level(g)\} \\ &\ge \level(fg),
	\end{align*} 
	while if $fg \in \D$, the equality $1/g = f/(fg)$ implies 
	\begin{align*}
	\max\{\level(f),\level(g)\} &\ge \level(g) \\&= \max\{\level(fg),\level(f)\} \\&\ge \level(fg).
	\end{align*}
	Moreover, in the first case, if $\level(g) \ne \level(f)$, then $\level(g) < \level(f)$ and $\level(fg) = \max\{\level(f),\level(g)\}$; a similar argument works in the second case, and (4) is proved.
	
	For (5) and (6), note that $\M(h) \subseteq \M(f) \cup \M(g)$.
	
	For (7), note that $f = (f/g) \cdot g$ and apply (4).
\end{proof}

\begin{expl}
	\label{strict_level_ineq}
	For an example where equality fails in Proposition \ref{D_level_lemma}(4), let $n \in \NN$ be at least 2, and set $f:= \log/\log_n \in \I \cap (\D \cdot \log)$ and $g:= 1/\log \in \D$.  Since $\sqrt{\log} > \log_n$, we have $\log > f > \sqrt{\log}$, so that $\level(f) = \level(g) = -1$; but $\level(fg) = -n$.
\end{expl}

What we are looking for is a notion of ``angular level'' $\alevel:\H^{>0} \into \ZZ$ that mirrors some of the properties of level, but where composition in Proposition \ref{D_level_lemma}(3) is replaced by the application of $\nu$, in the following sense: we want $\alevel(\nu_{\log}(f)) = \alevel(f) + 1$, $\alevel(\nu_{\exp}(f)) = \alevel(f)-1$ and $\alevel(\nu_{p_r}(f)) = \alevel(f)$.  The level does not do this, since for any $f \in \I \setminus \D\cdot \log$, the germ $\nu_{\log}(f)$ has level $-\infty$ (because $\theta_{\log}(f) \notin \D$ in this case).  However, the next two lemmas show that this is the only situation where the level is not adequate.

\begin{lemma}
\label{nu_mono}
Let $f,g \in \hplus$ and $0< r < s$ be real numbers.
\begin{enumerate}
\item If $f < g$, then $0 < \nu_{m_r}(f) < \nu_{m_r}(g)$, $0 < \nu_{p_r}(f) < \nu_{p_r}(g)$ and $\nu_{\log}(f) < \nu_{\log}(g)$.
\item If $f$ or $g$ is strictly decreasing and $sf \ge rg$ then, for any $t,u>0$, we have $\nu_{p_r \circ m_t}(g) < \nu_{p_s \circ m_u}(f)$ and $\nu_{m_t \circ p_r}(g) < \nu_{m_u \circ p_s}(f)$.
\end{enumerate}
\end{lemma}

\begin{proof}
Part (1) follows from Examples \ref{arg_function_expls} and Lemma \ref{image_of_real_domains}.  For part (2), assume that $sf \ge rg$ and let $t,u > 0$. Since $p_{1/r} \circ m_{1/t} > p_{1/s} \circ m_{1/u}$, it follows that
\begin{equation*}
\nu_{m_t \circ p_r}(g) = r g \circ p_{1/r} \circ m_{1/t} \le s f \circ p_{1/r} \circ m_{1/t} < sf \circ p_{1/s} \circ m_{1/u} = \nu_{m_u \circ p_s}(f),
\end{equation*}
if $f$ is strictly decreasing, while
\begin{equation*}
\nu_{m_t \circ p_r}(g) = r g \circ p_{1/r} \circ m_{1/t} < r g \circ p_{1/s} \circ m_{1/u} \le sf \circ p_{1/s} \circ m_{1/u} = \nu_{m_u \circ p_s}(f),
\end{equation*}
if $g$ is strictly decreasing; this proves the second inequality in (2).  The proof of the first inequality is similar.
\end{proof}

\begin{rmk}
Part (2) of the previous lemma is false if $f$ is increasing and $g=f$.
\end{rmk}

\begin{lemma}
\label{nu_lemma}
Let $f \in \D \cdot \log$.  Then there exists $u = u_{\log}(f) \in \H$ such that $u \sim 1$ and $$\nu_{\log}(f) \sim \frac f\log \circ \exp(p_1 \cdot u)\,.$$
\end{lemma}

\begin{proof}
As $f$ is fixed in this proof, we simply write $\rho$, $\theta$ and $\mu$ in place of $\rho_{\log}(f)$, $\theta_{\log}(f)$ and $\nu_{\log}(f)$ as in Example \ref{arg_function_expls}(2).  From the definition of $\rho$ and $\theta$, and since $f/\log \in \D$, $\arctan x = x + o(x)$ as $x \to 0$ and $\theta$ is bounded, we get
\begin{equation}
\label{sim1} \rho \sim \log
\end{equation}
and
\begin{equation}
\label{sim2} \theta \sim \frac{f}{\log}\ .
\end{equation}
By \eqref{sim1}, there exists $v \in \H$ such that $v(+\infty) = 1$ and
$\rho = \log \cdot\, v$.
Composing on the right with $\psi:= \rho^{-1}$ yields
$p_1 = (\log \circ \psi) \cdot (v \circ \psi)$,
so there exists $u \in \H$ such that $u(+\infty) = 1$ and
$p_{1} \cdot u = \log \circ \psi$.
Composing on the left with $\exp$ now gives
\begin{equation}
\label{sim3}
\exp(p_{1} \cdot u) = \psi.
\end{equation}
We get from \eqref{sim2} and \eqref{sim3} that
\begin{equation*}
\nu = \theta \circ \psi \sim \frac{f}{\log} \circ \psi = \frac{f}{\log} \circ  \exp(p_{1} \cdot u),
\end{equation*}
as required.
\end{proof}

The previous two lemmas imply that the level comes close to behaving like the angular level we are looking for:

\begin{cor}
\label{nu_cor}
\begin{enumerate}
\item For $r>0$, the map $\nu_{p_r}:\hplus \into \hplus$ is an increasing bijection with inverse $\nu_{p_{1/r}}$, and we have $\level\left(\nu_{p_r}(f)\right) = \level{f}.$
\item The map $\nu_{\log}:\hplus \into \Hsr$ is an increasing bijection with inverse $\nu_{\exp}$, and we have
\begin{equation*}
\level\left(\nu_{\log}(f)\right) = \begin{cases} \level\left(f/\log\right) + 1 &\text{if } f \in \D\cdot\log, \\ -\infty &\text{otherwise.} \end{cases}
\end{equation*}
\end{enumerate}
\end{cor}

\begin{proof}
Examples \ref{arg_function_expls}, Proposition \ref{D_level_lemma} and Lemmas \ref{nu_mono} and \ref{nu_lemma}.
\end{proof}

The previous corollary suggests a definition of angular level, at least for $f \in \D \cdot \log$ of level at least $0$; since then $\level(f/\log) = \level(f)$, we could take $\alevel(f) = \level(f)$ for such $f$ (or, more appropriately as we shall see below, as $\alevel(f) = \level(f) + 1$).  To figure out how to define angular level for the remaining $f$, we let $k \in \ZZ$ and set $$\D_{\ge k}:= \set{f \in \D:\ \level(f) \ge k}.$$
Note that, for $k \ge 0$, we have $\D_k \cdot \log = \D_k$, while $\D_{-1} \subsetneq \D_{-1} \cdot \log$ contains $\sqrt{\log}$, which is infinitely increasing.  Indeed, we have $$\H^{>0} = \I \setminus (\D_{-1} \cdot \log)\ \cup\ \D_{-1} \cdot \log\ \cup\ \D_{\ge 0}.$$  Since $\D_{-1} \cdot \log$ is convex and $\D_{\ge0} = \D_{\ge 0} \cdot \log$, it follows that

\begin{lemma}
  \label{Hsr_lemma}
  We have $(\Hone \cup \D) \setminus \D_{\ge 0}\ \subseteq\ \D_{-1} \cdot \log$. \qed
\end{lemma}


From Corollary \ref{nu_cor}, we obtain:

\begin{cor}
  \label{nu_cor2}
  We have $$\nu_{\log}\left(\hplus \setminus (\D_{\ge -1} \cdot \log)\right) = \left(0,\frac\pi2\right)_\H \setminus \D_{\ge 0} \subseteq\ \D_{-1} \cdot \log,$$ $$\nu_{\log}\left(\D_{-1} \cdot \log\right) = \D_0$$ and $$\nu_{\log}(\D_k) = \D_{k+1}, \quad\text{for } k \ge 0. \qed$$
\end{cor}

\begin{df}
\label{ang_level}
In view of the above, the \textbf{angular level} $\alevel:\hplus \into \{-1, 0, \dots \}$ is defined as \begin{align*}
\alevel(f):&= \max\set{-1, \level(\nu_{\log}(f))} \\ &= \begin{cases} -1 &\text{if } f \in \hplus \setminus (\D_{\ge -1} \cdot \log), \\  \level(\nu_{\log}(f)) &\text{if } f \in \D_{\ge -1} \cdot \log. \end{cases}
\end{align*}
\end{df}

From Proposition \ref{D_level_lemma} and Corollary \ref{nu_cor2}, we obtain:

\begin{prop}
\label{alevel_cor}
The angular level is decreasing and, for $f,g \in \hplus$, we have
\begin{enumerate}
\item $\alevel(\nu_{\log}(f)) = \alevel(f) + 1$;
\item if $f \in \D_{\ge -1} \cdot \log$, then $\alevel(f) = \max\{0,\level(f)+1\}$;
\item if $\alevel(f) \ge 1$, then $f \in \D$ and $f$ has level at least 0;
\item if $r,s>0$, then $\alevel(f) = \alevel(m_r \circ f \circ m_s)$;
\item for $r>0$, we have $\alevel(\nu_{p_r}(f)) = \alevel(f)$.
\end{enumerate}
\end{prop}

\begin{proof}
For part (1), since $\nu_{\log}(f) \in \Hone \cup \D$, we have from Corollary \ref{nu_cor}(2) that
\begin{align*}
\alevel(\nu_{\log}(f)) - 1 &= \level(\nu_{\log}(\nu_{\log}(f))) -1 \\ &= \max\{0, \level(\nu_{\log}(f))+1\} -1 \\ &= \max\{-1, \level(\nu_{\log}(f))\} \\ &= \alevel(f).
\end{align*}
Parts (2) and (3) follow similarly.  For part (4), note that if $f = g \cdot \log$, then
\begin{equation}\label{scaling_ineq}
  m_r \circ f \circ m_s = (m_r \circ g \circ m_s)\cdot (\log s + \log),
\end{equation}
which is greater than $(m_{r/2} \circ g \circ m_s) \cdot \log$ and less than $(m_{2r} \circ g \circ m_s) \cdot \log$.  Thus, if $g \in \D_k$, then $m_{r/2} \circ g \circ m_s$ and $m_{2r} \circ g \circ m_s$ belong to $\D_k$ and, since the latter is an interval, part (4) follows in this case.  If $g \notin \D$ then, by equation \eqref{scaling_ineq},  neither is $m_r \circ f \circ m_s$, so part (4) also follows in this case.

For part (5) and $r>0$, note that $\nu_{p_r}(\log) = \log$, which implies that $\nu_{p_r}(\D_{-1} \cdot \log) = \D_{-1} \cdot \log$.
\end{proof}

\begin{expls}
	\label{al_expls}
	\begin{enumerate}
		\item If $f \in \Hone$, then $\level(f) = -\infty$ and $f \in \D_{\ge 1}\cdot\log$, so that $\alevel(f) = 0$.  Thus, if $\frU \subseteq \LL$ is a standard power domain, then $\lim_{x \to +\infty}f_{\frU,\arg}(x) = \frac\pi2$, so that $\alevel\left(f_{\frU,\arg}\right) = 0$.
		\item If $f \in \I$, then $\alevel(f) \in \{-1,0\}$; in particular, if $f \succ \log$, then $\alevel(f) = -1$.
		\item If $f \in \D$, then $\alevel(f) \ge 0$; in particular, if $f \prec \frac1\log$, then $\alevel(f) = \level(f)+1$.  Thus, $\alevel(1/\exp_{k-1}) = k$, for $k \in \NN$.
		\item For $S:= \set{z \in \CC:\ \re z > 0,\ |\im z| < \pi}$, we have $f_{S,\im} = \pi$ and $f_{S,\arg} \sim \arctan(\pi/x)$.  Since $\arctan(\pi/x) \sim \pi/x \prec 1/\log x$ as $x \to +\infty$, it follows that $\level(f_{S,\arg}) = 0$ and, by the previous example, that $\alevel(f_{S,\arg}) = 1$.
	\end{enumerate}
\end{expls}

The angular level provides us with a way to assign size to certain domains in $\LL$:

\begin{df}
	\label{multi-quadratic_domain}
	For $k \in \NN p \{-1\}$, we set $$\Hang_k := \set{h \in \hplus:\ \alevel(h) = k};$$  in particular, $\Hang_{-1} = \hplus \setminus (\D_{\ge -1} \cdot \log)$ and $\Hang_0 = \D_{-1} \cdot \log$.
	A domain $\frU \subseteq \LL$ is called a \textbf{$k$-domain} if there exist $f,g \in \Hang_k$ such that $\frU_f \subseteq \frU \subseteq \frU_g$.
\end{df}

It follows from Example \ref{al_expls}(1) that every standard power domain is a $0$-domain.

\subsection*{Angular level for maps} 
Based on the notion of angular level, we can now define a notion of angular level for maps on $\LL$.

\begin{df}
	\label{complex level}
Let $\lambda \in \ZZ$, $\frU,\frV \subseteq \LL$ be domains and $\varphi:\frU \into \frV$.  We say that
$\varphi$ has \textbf{angular level $\lambda$} if, for $k \ge \lambda -1$ and every $k$-domain $\frU' \subseteq \frU$, the image $\varphi(\frU)$ is a $(k-\lambda)$-domain.
\end{df}

\begin{expl}
\label{levelled_maps}
It follows from Examples \ref{arg_function_expls}(1)--(3) and Proposition \ref{alevel_cor} that $\m_r, \p_r:\LL \into \LL$ have angular level 0, for $r>0$, that $\llog:\H_\LL(1) \into S_\LL(\pi/2)$ has angular level $-1$, and that $\eexp:S_\LL(\pi/2) \into H_\LL(1)$ has angular level 1.  Also, by Example \ref{arg_function_expls}(4), the map $\t_a:\LL \into \LL$ has angular level 0, for $a>0$; it is in order to include maps like the latter (which are analytic continuations of definable maps!) that the notions of $k$-domain and angular level for such maps are defined in such  generality.
\end{expl}

Finally, we can define the kind of analytic continuation property we are ultimately interested in.

\begin{df}
\label{weakly_extending}
Let $\eta \in \NN$ and $\lambda \in \ZZ$ be such that $\lambda \le \eta$.
\begin{enumerate}
\item Let $\frU, \frV \subseteq \LL$ and $\varphi:\frU \into \frV$.  We call $\varphi$ an \textbf{$(\eta,\lambda)$-map} if the following hold:
\begin{renumerate}

\item $\frU$ is an $(\eta-1)$-domain, $\frV$ is an $(\eta-\lambda-1)$-domain and $\varphi$ is biholomorphic of angular level $\lambda$;
\item for $k \ge \eta-1$ and every $(k-\lambda)$-domain $\frV' \subseteq \frV$, there exists a $k$-domain $\frU' \subseteq \frU$ such that $\varphi(\frU') \subseteq \frV'$;
\item there exist $h_1,h_2 \in \I$ of level $\lambda$ such that
\begin{equation*}
h_1(|x|) \le |\varphi(x)| \le h_2(|x|),
\end{equation*}
for all sufficiently large $x \in U$.
\end{renumerate}
\item Let $f \in \I$.  The function $f$ is called \textbf{$(\eta,\lambda)$-extendable} if there exists an $(\eta,\lambda)$-map $\f:\frU \into \frV$ that extends $f$.  In this situation, we refer to $\f$ as an \textbf{$(\eta,\lambda)$-extension} of $f$.
\end{enumerate}
\end{df}

\begin{expl}
	\label{extendable_expl}
	By Example \ref{levelled_maps}, for $r>0$ the maps $m_r$ and $p_r$ are $(0,0)$-extendable, $\exp$ is $(1,1)$-extendable and $\log$ is $(0,-1)$-extendable.
\end{expl}

\begin{lemma}
\label{weak_ext_comp}
If $f_i \in \I$ is $(\eta_i,\lambda_i)$-extendable, for $i=1,2$, then $f_2 \circ f_1$ is $(\eta,\lambda_1+\lambda_2)$-extendable, where $\eta:= \max\{\eta_1,\eta_2+\lambda_1\}$.
\end{lemma}

\begin{proof}
Let $\f_i: \frU_i \into \frV_i$ be an $(\eta_i,\lambda_i)$-extension of $f_i$, for each $i$.  If $\frV_1 \subseteq \frU_2$ then, by definition, $\f_2 \circ \f_1$ is an $(\eta_1, \lambda_1+ \lambda_2)$-extension of $f_2 \circ f_1$.  So assume that $\frV_1 \nsubseteq \frU_2$, and let $\frU_1'$ be an $(\eta_2+\lambda_1-1)$-domain such that $\f_1(\frU_1') \subseteq \frU_2$ (which exists by part (1b) of Definition \ref{weakly_extending}).  Then, again by definition, $\f_2 \circ (\f_1 \rest{\frU_1'})$ is an $(\eta_2+\lambda_1,\lambda_1+\lambda_2)$-extension of $f_2 \circ f_1$.
\end{proof}

\subsection*{Brief outline of how we obtain extendability}
Given $f \in \I$, we set $\eta:= \max\{0,\eh(f)\}$ and $\lambda:= \level(f) \le \eta$; our goal (see Continuation Theorem \ref{I_ext_prop} below) is to show that $f$ is $(\eta,\lambda)$-extendable.  
To prove this, we may assume $f \in \E$ by Proposition \ref{1.11} and Lemma \ref{weak_ext_comp}, because we already know that $\log$ is $(0,-1)$-extendable.  For $f \in \E$, we will proceed by induction on $\eh(f)$; in the inductive step, we first obtain extendability for principle monomials, which is straightforward from the inductive hypothesis and Lemma \ref{weak_ext_comp}, because we already know that $\exp$ is $(1,1)$-extendable.  Then we show that multiplying by a unit preserves the extension property of the principle monomials, which requires us to track additional properties introduced in the next section.

\section{Continuity and angular orientation properties}
\label{opposing}

A crucial step in proving the Continuation Theorem \ref{I_ext_prop}  involves simultaneously establishing two ``opposing'' continuity properties for the analytic continuations of units in $\E$ (which are \textit{$\D$-Lipschitz}, see Proposition \ref{u_small} below) and of infinitely increasing germs in $\E$ (which are \textit{expansive}, see Proposition \ref{exp_ext_prop} below). 

The upshot of these additional properties is that (the analytic continuation of) an infinitely increasing $f \in \E$ and the product of $f$ with a unit $u$ in $\E$ are ``almost the same''; for instance, the analytic continuation of $fu$ is still expansive.  As a result, establishing the extendability properties of $f$ reduces to establishing those of the leading monomial of $f$, which can be handled using the inductive hypothesis.

However, the distance $d$ on $\LL^2$ makes $\LL$ into a \textit{multiplicative} metric space, which makes the corresponding definitions of continuity and differentiability unfamiliar to work with.  We will therefore often work in one of the following two charts:

\subsection*{The standard chart}

We denote by $\pi_0$ the injective restriction of $\pi$ to $S_\LL(\pi)$, and we call $\pi_0:S_\LL(\pi) \into \CC \setminus (-\infty,0]$ the \textbf{standard chart} on $\LL$. Recall that, for $h \in \hplus$, we have $$h < \pi \quad\text{if and only if}\quad \pi_0^{-1}(\pi_0(U_h)) = U_h.$$ 
For $\kappa \ge 0$ and $U \subseteq \CC \setminus (-\infty,0]$, we call $U$ a \textbf{$\kappa$-domain} if and only if $\pi_0^{-1}(U)$ is a $\kappa$-domain.


\subsection*{The $\Log$-chart}

Recall that $\Log:\LL \into \CC$ is the biholomorphic map $$\Log x:= \log|x| + i \arg x$$ with compositional inverse $\Exp:\CC \into \LL$, and that $$d(x,y):= |\Log x - \Log y|$$ defines the corresponding distance function on $\LL$, making $(\LL,d)$ into a multiplicative metric space.  Given a domain $\frU \subseteq \LL$ and a map $\varphi:\frU \into \LL$, we denote by $\bar\varphi:\Log(\frU) \into \CC$ the holomorphic map $$\bar\varphi:= \Log \circ \varphi \circ \Exp.$$  If, in addition, $\varphi$ is holomorphic, then the derivative $\varphi':\frU \into \LL$  of $\varphi$ in the sense of the $\Log$ chart is given by $$\varphi' = \Exp \circ \left(\bar\varphi\right)' \circ \Log.$$  Note that $\re\bar\varphi = \log|\varphi| \circ \Exp$ and, if $\varphi$ is holomorphic, that $\bar{\varphi'} = \left(\bar\varphi\right)'$.

\begin{nrmks}
	\label{sr-domain_shrink}
	Let $h \in \H^{>0}$.  
	\begin{enumerate}
		\item The set $U:= \Log(\frU_h)$ is (the germ of) a definable standard domain such that $\bar h:= f_{U,\im} = h \circ \exp \in \H^{>0}$; in particular, if $V:= \Log(\frU_{h/2})$, we have $f_{V,\im} = (h/2) \circ \exp = \bar h/2$.
		\item If $h$ is bounded, the derivative of $h(x)$ approaches 0 as $x \to +\infty$, so we have $B\left(z,\min\{1,\bar h(\re z)/3\}\right) \subseteq \Log(\frU_h)$ for all $z \in \Log(\frU_{h/2})$ with sufficiently large real part.  
		On the other hand, if $h$ is infinitely increasing, then $B(z,1) \subseteq \Log(\frU_h)$ for all $z \in \Log(\frU_{h/2})$ with sufficiently large real part.  Finally, for $x \in \LL$ and $z = \Log x$, we have $\bar h(\re z) = h(|x|)$.  It follows that $$B\left(x,\min\left\{1,h(|x|)/3\right\}\right) \subseteq \frU_{h}$$ for all sufficiently large $x \in \frU_{h/2}$.
	\end{enumerate}
\end{nrmks}

\begin{expl}
	\label{lipschitz_expl}
	Let $\frU \subseteq \LL$ be a domain and $\varphi:\frU \into \LL$, and let $a > 0$ be a real constant.  Recall that $\varphi$ is \textbf{$a$-Lipschitz} if
	\begin{equation}\label{lipschitz_eq}
		d(\varphi(x),\varphi(y)) \le a \cdot d(x,y) \quad\text{for } x,y \in \frU. 
	\end{equation}
	Note that $\varphi$ is $a$-Lipschitz if and only if $\left|\bar\varphi(z) - \bar\varphi(w)\right) \le a \cdot |z-w|$ for $z,w \in \Log \frU$, that is, if and only if $\bar\varphi$ is $a$-Lipschitz.
\end{expl}

We now define two continuity properties of holomorphic maps on $\LL$ inspired by the previous example, as well as an angular orientation property that will come in handy later.

\begin{df}
	\label{opposing_df}
	Let $\frU \subseteq \LL$ be a domain and $\varphi:\frU \into \LL$.
	\begin{enumerate}
		\item The map $\varphi$ is \textbf{$\D$-Lipschitz} if there exists $\rho \in \D$ such that 
		\begin{equation*}
		\label{D-lipschitz_eq}
		d(\varphi(x),\varphi(y)) \le \rho(\min\{|x|,|y|\}) \cdot d(x,y) \quad\text{for } x,y \in \frU;  
		\end{equation*}
		in this situation, we say more precisely that $\varphi$ is \textbf{$\rho$-Lipschitz}.
		\item The map $\varphi$ is \textbf{expansive} if there exists a real $a>0$ such that 
		\begin{equation*}
		\label{expansive_eq}
		d(\varphi(x),\varphi(y)) \ge a \cdot d(x,y) \quad\text{for } x,y \in \frU;  
		\end{equation*}		
		in this situation, we say more precisely that $\varphi$ is \textbf{$a$-expansive}.
		\item The map $\varphi$ is \textbf{angle-positive} if, for all $x \in \frU$, $\arg x$ and $\arg\varphi(x)$ have the same sign.  The map $\varphi$ is \textbf{angle-negative} if, for all $x \in \frU$, $\arg x$ and $\arg\varphi(x)$ have opposite sign.
	\end{enumerate}
\end{df}

\begin{rmks}
	We shall use the following observations without explicit mention:
	in the setting of the previous definition, the map $\varphi$ is $\rho$-Lipschitz if and only if 
	\begin{equation*}
	\label{D-lipschitz_eq_2}
	\left|\bar\varphi(z)-\bar\varphi(w)\right| \le (\rho \circ \exp)(\min\{\re z, \re w\}) \cdot |z-w|  
	\end{equation*}
	for $z,w \in \Log(\frU)$.  Since $\bar{1/\varphi} = -\bar\varphi$, it follows that $\varphi$ is $\rho$-Lipschitz if and only if $1/\varphi$ is.
	
	Next, $\varphi$ is $a$-expansive if and only if
	\begin{equation*}
	\label{expansive_eq_2}
	\left|\bar\varphi(z)-\bar\varphi(w)\right| \ge a \cdot |z-w|  
	\end{equation*}	
	for $z,w \in \Log(\frU)$.  In particular, if $\varphi$ is expansive, then $\varphi$ is injective; the converse does not hold, as $\eexp$ is injective on $S_\LL(\pi/2)$ but is not expansive there (note that $\bar\eexp = \exp$).
	
	Finally, $\varphi$ is angle-positive if and only if $1/\varphi$ is angle-negative, and in both cases we have $\varphi((0,\infty)) \subseteq (0,\infty)$.  Conversely, if $\frU$ is a real domain, $\phi$ is continuous and injective and $\phi((0,\infty)) \subseteq (0,\infty)$, then $\phi$ is either angle-positive or angle-negative (because $(0,\infty)$ topologically separates $\frU\setminus(0,\infty)$).
\end{rmks}

\begin{expls}
	\label{expansive_expl}
	\begin{enumerate}
		\item For $r>0$, the maps $\m_r$ and $\p_r$ are expansive, because $\bar\m_r$ is translation by $\log r$, while $\bar\p_r$ is multiplication by $r$.  Moreover, $\eexp$ is expansive on $S_\LL(\pi/2) \cap H_\LL(1)$, as the next lemma shows.
		\item For $r>0$, both $\m_r$ and $\p_r$ are angle-positive, and both $\eexp$ and $\llog$ are angle-positive as well.
	\end{enumerate}
\end{expls}

\begin{lemma}
	\label{expansive_lemma}
	The map $\eexp$ is expansive on $S_\LL(\pi/2) \cap H_\LL(1)$.
\end{lemma}

\begin{proof}
	Since $\bar\eexp = \exp$, we work with $\exp$ on the set $$S_{>0}:= \Log\left(S_\LL(\pi/2) \cap H_\LL(1)\right) = \set{z \in \CC:\ \re z > 0, |\im z| < \frac\pi2}.$$
	For $z,w \in \CC$, we have
	\begin{equation*}
	\frac{e^z-e^w}{z-w} = e^w \left(\frac{e^{z-w}-1}{z-w}\right) = e^w \left(1 + \sum_{n=2}^\infty \frac{(z-w)^{n-1}}{n!}\right).
	\end{equation*}
	Let $R>0$ be such that $\sum_{n=2}^\infty R^{n-1}/n! = 1/2$; it follows that, if $z,w \in S_{>0}$ with $|z-w| < R$, we have
	\begin{equation*}
	\left|\frac{e^z-e^w}{z-w}\right| \ge \frac{|e^w|}2 > \frac12.
	\end{equation*}
	On the other hand, set $A:= 1-e^{-R/2} > 0$ and let $B>0$ be the distance between 1 and the complement in $\CC$ of the sector $S(R/2)$.  Write $z = x+iy$ and $w = u+iv$, and assume that $z,w \in S_{>0}$ with $|z-w|\ge R$.  Then $|x-u| \ge R/2$ or $|y-v| \ge R/2$; in the former case, we get $|e^{z-w}-1| \ge A$, so that
	\begin{equation*}
	\left|\frac{e^z-e^w}{z-w}\right| \ge \frac{A}R;
	\end{equation*}
	in the latter case, since $|y-v| < \pi$, it follows that $e^{z-w}$ belongs to the complement in $\CC$ of the sector $S(R/2)$, so we get
	\begin{equation*}
	\left|\frac{e^z-e^w}{z-w}\right| \ge \frac{B}R
	\end{equation*}
	in this case.  This proves the lemma.
\end{proof}

Three of these properties are preserved under composition:

\begin{lemma}
	\label{comp_expansive}
	Let $\frU,\frV \subseteq \LL$ be domains and let $\varphi:\frU \into \frV$ and $\psi:\frV \into \LL$. 
	\begin{enumerate}
		\item Let $\rho,\sigma \in \D$.  If $\varphi$ is $\rho$-Lipschitz and $\psi$ is $\sigma$-Lipschitz, then $\psi\circ\varphi$ is $\rho\sigma$-Lipschitz.
		\item Let $a,b>0$.  If $\varphi$ is $a$-expansive and $\psi$ is $b$-expansive, then $\psi\circ\varphi$ is $ab$-expansive.
		\item If both $\varphi$ and $\psi$ are angle-positive, then so is $\psi \circ \varphi$.
	\end{enumerate} 
\end{lemma}

\begin{proof}
	For distinct $x,y \in \frU$, we have $$\frac{d((\psi \circ \varphi)(x), (\psi\circ\varphi)(y))}{d(x,y)} = \frac{d((\psi \circ \varphi)(x), (\psi\circ\varphi)(y))}{d(\varphi(x),\varphi(y))} \cdot \frac{d(\varphi(x),\varphi(y))}{d(x,y)},$$ so parts (1) and (2) follow.  Part (3) is straightforward.
\end{proof}

Crucial for us is the following observation that ``expansive trumps $\D$-Lipschitz'' under multiplication:

\begin{lemma}
	\label{expansive_unit}
	Let $\frU \subseteq \LL$ be a domain at $\infty$ and $\varphi,\psi:\frU \into \LL$ be such that $\varphi$ is expansive and $\psi$ is $\D$-Lipschitz.  Then there exists $R>0$ such that $\varphi\psi$ is expansive on $\frU \cap H_\LL(R)$.
\end{lemma}

\begin{proof}
	Let $\rho \in \D$ and $a>0$ be such that $\varphi$ is $a$-expansive and $\psi$ is $\D$-Lipschitz.  Let $R>0$ be such that $\rho(t) \le a/2$ for $t > R$.  Then, for distinct $x,y \in \frU \cap H_\LL(R)$, we have
	\begin{align*}
		\frac{d(\varphi\psi(x),\varphi\psi(y))}{d(x,y)} &= \frac{|\Log(\varphi\psi(x)) - \Log(\varphi\psi(y))|}{d(x,y)} \\
		&\ge \frac{|\Log\varphi(x) - \Log\varphi(y)| - |\Log\psi(x) - \Log\psi(y)|}{d(x,y)} \\
		&\ge a - \rho(\min\{|x|,|y|\}) \\
		&\ge \frac a2,
	\end{align*}
	as required.
\end{proof}

\begin{df}\label{unit_at_infty}
	Let $\frU \subseteq \LL$ be a domain at $\infty$ and $\varphi:\frU \into \LL$ be a map.  We call $\varphi$ a \textbf{unit at $\infty$} if $d(\varphi(x),1) \to 0$ as $|x| \to \infty$.
\end{df}

\begin{expl}
	\label{unit_expl}
	Let $u \in \Pc{R}{X}$ be such that $u(0) = 1$.  Then, by Example \ref{basic_extensions_2}(3), $u$ has an analytic continuation $$\u:\pi^{-1}(B(0,a)) \into \LL,$$ for some sufficiently small $a>0$, such that $\u \circ \p_{-1}:H_\LL(1/a) \into \LL$ is a unit at $\infty$.
\end{expl}

Finally, here is what happens when we combine all four properties:

\begin{lemma}
	\label{mult_by_unit_pos_exp}
	Let $k \ge -1$, $\frU \subseteq \LL$ a $k$-domain, $\frV \subseteq \LL$ a $(-1)$-domain and $\phi:\frU \into \frV$ be continuous, surjective, expansive and angle-positive.  Let also $\psi:\frU \into \LL$ be a $\D$-Lipschitz unit at $\infty$ such that $\psi(\frU \cap (0,\infty)) \subseteq (0,\infty)$.  Then there exists $r>0$ such that:
	\begin{enumerate}
		\item $\phi\psi\rest{\frU \cap H_{\LL}(r)}$ is expansive and angle-positive;
		\item for $x \in \frU \cap H_{\LL}(r)$, we have $$\frac12|\arg\phi(x)| \le |\arg(\phi\psi)(x)| \le \frac32|\arg\phi(x)|.$$
	\end{enumerate}
\end{lemma}

\begin{proof}
	(1) It follows from Lemma \ref{expansive_unit} that, after replacing $\frU$ by $\frU \cap H_\LL(r)$ for some $r>0$ if necessary, the map $\phi\psi$ is expansive; in particular, $\phi\psi$ is injective.
	Moreover, since $\psi$ is a unit at $\infty$, there exists $A>0$ such that, for all $x \in \frU$ with $|x|$ sufficiently large, we have $|\arg\psi(x)| \le A$.  Since $\frV = \phi(\frU)$ is a $(-1)$-domain, there exist arbitrarily large $x \in \frU$ such that $\arg x > 0$ and $\arg\phi(x) > A$; in particular, $\arg(\phi\psi(x)) > 0$ for arbitrarily large $x \in \frU$.  On the other hand, since $\phi\psi\rest{\RR}$ is continuous and unbounded, there exist $r,s>0$ such that $\phi\psi((r,\infty)) = (s,\infty)$.  Since $\phi\psi$ is injective, it follows that $\phi\psi\rest{\frU \cap H_{\LL}(r)}$ is angle-positive as well.
	
	(2) Note that $\psi^2$ is also a unit at $\infty$.  Moreover, since $\bar{\psi^2} = 2\bar\psi$, the map $\psi^2$ is also $\D$-Lipschitz.  Therefore, applying part (1) with $\psi^2$ in place of $\psi$ we get, for sufficiently large $x \in \frU$ with $\arg x > 0$, that $$0 < \arg(\phi\psi^2(x)) = \arg\phi(x) + 2\arg\psi(x);$$ that is, 
	\begin{equation}
	\label{unit_lower_bound}
	\arg\psi(x) \ge -\arg\phi(x)/2,
	\end{equation}
	for sufficiently large $x \in \frU$ with $\arg x > 0$.  Similarly, since $1/\psi$ is also a unit at $\infty$, and since $\bar{1/\psi} = -\bar\psi$ we obtain, from \eqref{unit_lower_bound} with $1/\psi$ in place of $\psi$, that
	\begin{equation}
	\label{unit_upper_bound}
	\arg\psi(x) \le \arg\phi(x)/2, 
	\end{equation}
	for sufficiently large $x \in \frU$ with $\arg x > 0$.  The two inequalities \eqref{unit_lower_bound} and \eqref{unit_upper_bound} together prove part (2) for $\arg x > 0$; the case $\arg x < 0$ follows by a symmetric argument.
\end{proof}

\section{Units in $\E$ are $\D$-Lipschitz}
\label{extendable}

The aim of this section is to show---see Proposition \ref{u_small} below---that units $u \in \E$ have $\D$-Lipschitz analytic continuations that are units at $\infty$, provided that all large principle monomials $m$ of $u$ are $(\mu,\mu)$-extendable, where $\mu = \level(m)$.  (This last assumption will be obtained from our inductive hypothesis in the proof of Proposition \ref{exp_ext_prop}.)

The first step towards establishing Proposition \ref{u_small} involves first describing the complex-valued analytic continuations of small germs in $\E$, see Proposition \ref{derivative_of_small} below.

We define $\arg:\CC \into (-\pi,\pi]$ to be the standard argument in $\CC \setminus (-\infty,0]$ and equal to $\pi$ on $(-\infty,0]$.

\begin{nrmk}
	\label{addition_rmk}
	For $z,w \in H(0)$, we have $$|\arg(z+w)| \le \max\{|\arg z|, |\arg w|\}.$$  
\end{nrmk}

\begin{lemma}
\label{unit_arg_lemma}
Let $u \in \Pc{R}{X_1, \dots, X_k}$ be such that $u(0) = 1$.  Then there exists $B>0$ such that $$|\arg u(z_1, \dots, z_k)| \le B \cdot \max\{|\arg z_1|, \dots, |\arg z_k|\}$$
for sufficiently small $z = (z_1, \dots, z_k) \in \CC^k$.
\end{lemma}

\begin{proof}
Writing $z_i = x_i + i y_i$, $x = (x_1, \dots, x_k)$ and $y = (y_1, \dots, y_k)$, and since $u$ has real coefficients, there are $R,I \in \Pc{R}{X,Y}$ such that $u(z) = R(x,y) + iI(x,y)$.  Since the restriction of $u$ to $\RR^k$ is real-valued, we have $I(X,0) = 0$, so there are $I_i \in \Pc{R}{X,Y}$ such that $I(X,Y) = Y_1I_1(X,Y) + \cdots + Y_kI_k(X,Y)$.  It follows that there are $A_1, \dots, A_k>0$ such that
\begin{equation*}
|\im u(z)| \le A_1|\im z_1| + \cdots + A_k|\im z_k|
\end{equation*}
for sufficiently small $z \in \CC^k$.  On the other hand, for $z \in \CC$ we have $\im z = |z| \cdot \sin(\arg z)$.  Since $u(0) = 1$, it follows for sufficiently small $z \in \CC^k$ that
\begin{align*}
|\sin(\arg u(z))| &= \frac{|\im u(z)|}{|u(z)|} \\ &\le 2(A_1|\im z_1| + \cdots + A_k|\im z_k|) \\ &\le 2(A_1 |\sin(\arg z_1)| + \cdots + A_k |\sin(\arg z_k)|).
\end{align*}
Since $|\sin t| \le |t|$ for $t \in \RR$, the lemma follows.
\end{proof}

\begin{rmk}
	Let $G \in \Pc{R}{X_1, \dots,X_k}$ be such that $G(0) = 0$, and let $a>0$ be sufficiently small and $g$ be the complex-valued holomorphic map defined by $G$ on $B(0,a)^k$.  Then $g\circ\pi_k$ is a complex-valued holomorphic map on $\pi_k^{-1}\left((B(0,a)\setminus\{0\})^k\right)$, where $\pi_k:\LL^k \into (\CC\setminus\{0\})^k$ is defined by $\pi_k(x_1, \dots, x_k):= (\pi(x_1), \dots, \pi(x_k))$.  Moreover, there is a $K>0$ such that $$\left|(g \circ \pi_k)(x_1, \dots, x_k)\right| \le K \max\{|x_1|, \dots, |x_k|\}$$ for sufficiently small $x_1, \dots, x_k \in \LL$:  this follows, because the assumption $G(0) = 0$ implies that there exist $l \in \NN$, nonzero $\alpha_1, \dots \alpha_l \in \NN^k$ and $G_1, \dots, G_k \in \Pc{R}{X_1, \dots, X_k}$ such that $G = X^{\alpha_1} G_1 + \cdots + X^{\alpha_l} G_l$.
\end{rmk}

\begin{prop}
\label{derivative_of_small}
Let $G \in \Pc{R}{X_1, \dots,X_k}$ be such that $G(0) = 0$, let $k,n \in \NN$ and $m_1, \dots, m_k \in \M_n \setminus \M_{n-1}$ be large, and set $s:= G(1/m_1, \dots, 1/m_k)$.  Assume that each $m_i$ is  $(n,n)$-extendable.  Then:
\begin{enumerate}
\item if $n=0$, there exist $R,B>0$ and an analytic continuation $\ss: \frU:= H_\LL(R) \into \CC$ of $s$ such that $|\ss(x)| \le B/|x|$ for $x \in H_\LL(R)$;
\item if $n \ge 1$, there exist an $(n-1)$-domain $\frU \subseteq H_\LL(1)$, an analytic continuation $\ss:\frU \into \CC$ of $s$ and an $e \in \I$ of level $n$ such that $|\ss(x)| \le 1/e(|x|)$ for $x \in \frU$.
\end{enumerate}
\end{prop}

\begin{rmk}
	By Lemma \ref{f_going_to_0}, every small germ in $\P\E_n$ is of the form $s$ as in the previous proposition.
\end{rmk}

\begin{proof}
(1) Assume $n=0$; without loss of generality, we may assume that $k=1$ and $m_1 = p_1$.  So part (1) follows in this case from the previous remark.

(2) Assume $n \ge 1$, and choose an $(n-1)$-domain $\frU \subseteq \LL$ and $(n,n)$-extensions $\m_i:\frU \into \frV_i$ of $m_i$, for $i=1, \dots, k$.  Let $e \in \I$ be of level $n$ such that $|\m_i(x)| \ge e(|x|)$ for all sufficiently large $x \in \frU$.  By the previous remark, there exists $K>0$ such that $|(g \circ \pi)(x_1, \dots, x_k)| \le K\max\{|x_1|, \dots, |x_k|\}$ for sufficiently small $x_1, \dots, x_k \in \LL$.  Hence $\ss := (g \circ \pi) (1/\m_1, \dots, 1/\m_k)$ is an analytic continuation of $s$ on $\frU$ such that $|\ss(x)| \le K/e(|x|)$ for sufficiently large $x \in \frU$. 
\end{proof}

Next, we fix $\eta \in \NN$ and let $u \in \E_\eta$ be such that $\lim_{x \to +\infty} u(x) = 1$; it follows that $\log \circ\, u$ belongs to $\E_\eta$ and is small.  Hence, by Lemma \ref{1.5}, there are small $u_n, v_n \in \P\E_n$, for $n=0, \dots, \eta$, such that $$u = 1+ u_0 + \cdots + u_\eta \quad\text{and}\quad \log \circ\, u = v_0 + \cdots + v_\eta.$$  By Lemma \ref{f_going_to_0}, for each $n$, there are large $m_{n,1}, \dots, m_{n,k_n} \in \M_n \setminus \M_{n-1}$ and $G_n, H_n \in \Pc{R}{X_1, \dots, X_{k_n}}$ such that $G_n(0) = H_n(0) = 0$ and $u_n = G_n(1/m_{n,1}, \dots, 1/m_{n,k_n})$ and $v_n = H_n(1/m_{n,1}, \dots, 1/m_{n,k_n})$.

Below, we consider the assumption
\begin{itemize}
\item[$(\ast)_\eta$] every large $m \in \M_n \setminus \M_{n-1}$ is $(n,n)$-extendable, for $n=0, \dots, \eta$.
\end{itemize}

\begin{cor}
\label{unit_lemma}
Assume that $(\ast)_\eta$ holds.  Then there exists an $(\eta-1)$-domain $\frU$ and an analytic continuation $\u: \frU \into \LL$ of $u$ that is a unit at $\infty$.
\end{cor}

\begin{proof}
By Proposition \ref{derivative_of_small}, for $n= 0, \dots, \eta$, there exist an $(n-1)$-domain $\frU_n$ and an analytic continuation $\uu_n:\frU_n \into \CC$ of $u_n$ such that $|\uu_n(z)| \le 1/e_n(|z|) \le 1/|z|$, where $e_n \in \I$ is of level $n$.  Define $\uu:\frU_\eta \into \CC$ by $$\uu(z):= 1 + \uu_0(z) + \cdots + \uu_\eta(z);$$  then $\uu(z) \to 1$ as $|z| \to \infty$ in
$\frU_\eta$, so we define $\u:= \pi_0^{-1} \circ \uu$ on $\frU_\eta \cap H_\LL(a)$, for some sufficiently large $a>0$.
\end{proof}

We now fix an analytic continuation $\u:\frU \into \LL$ that is a unit at $\infty$, as obtained from Corollary \ref{unit_lemma}.  It remains to show that $\u$ is $\D$-Lipschitz.

\begin{lemma}
\label{u_Cauchy}
Assume that $(\ast)_\eta$ holds.  Then, for sufficiently small $h \in \Hang_{\eta-1}$, there exists $\sigma \in \D$ such that
\begin{equation}
\label{pre-lipschitz} d(\u(x), \u(y)) \le \sigma(|x|) \cdot d(x,y)
\end{equation}
for $x \in \frU_{h/2}$ and $y \in B\left(x,\min\left\{1,h(|x|)/3\right\}\right)$.
\end{lemma}

\begin{proof}
We work in the $\Log$-chart and show that
\begin{itemize}
	\item[$(\ast)$] for sufficiently small $h \in \Hang_{\eta-1}$, there exists $\rho \in \D$ such that $|\bar\u(z+w) - \bar\u(z)| \le \rho(\re z) \cdot |w|$ for $z \in \Log(\frU_{h/2})$ and $w \in B\left(z,\min\left\{1,h(|x|)/3\right\}\right)$;
\end{itemize} 
taking $\sigma:= \rho \circ \log$ then finishes the proof of the lemma.
We distinguish two cases: \medskip

\noindent\textbf{Case 1:} $\eta = 0$; note that $\Hang_{\eta-1} \subset \I$, $u = 1+u_0$ and $\log \circ\, u = v_0$ in this case.
By Proposition \ref{derivative_of_small}(1), there exist $R,B>0$ and an analytic continuation $\vv_0: H_\LL(R) \into \CC$ of $\log \circ \, u$ such that $|\vv_0(x)| \le B/|x|$ for $x \in H_\LL(R)$.  Since $\bar\u\rest{\RR} = v_0 \circ \exp$, it follows from the identity theorem for holomorphic functions that $\bar\u = \vv \circ \Exp$ on $\Log(\frU) \cap H(\log R)$; in particular, we have
\begin{equation}
\label{eta_0_bound}
|\bar\u(z)| \le \frac B{\exp(\re z)} \quad\text{for } z \in \Log(\frU) \text{ with } \re z > \log R.
\end{equation}
Now let $h \in \Hang_{\eta-1}$ be such that $\frU_h \subseteq U$.  Let also $z \in \Log(\frU_{h/2})$ with $\re z > 1+\log R$ be sufficiently large so that $B(z,1) \subseteq \Log(\frU_h)$ (see Remark \ref{sr-domain_shrink}(2)), and let $w \in B(0,1)$.  From the Taylor series of $\bar\u$ at $z$, we get
\begin{equation*}
\bar\u(z+w) - \bar\u(z) = \sum_{n=1}^\infty \frac{\bar\u^{(n)}(z)}{n!} w^n = w \cdot \sum_{n=0}^\infty \frac{\bar\u^{(n+1)}(z)}{(n+1)!} w^n.
\end{equation*}
By the Cauchy estimates and \eqref{eta_0_bound}, we have $\left|\bar\u^{(n)}(z)\right| \le \frac{n! B}{\exp(\re z-1)}$, so for $|w| \le 1/2$, we get
\begin{equation*}
|\bar\u(z+w) - \bar\u(z)| \le |w| \cdot \frac{2B}{\exp(\re z-1)} = |w| \cdot \frac{2Be}{\exp(\re z)},
\end{equation*}
so we can take $\rho(t):= \frac{2Be}{\exp t}$ for sufficiently large $t > 1+\log R$.
\medskip

\noindent\textbf{Case 2:} $\eta > 0$.  By Proposition \ref{derivative_of_small}(2), for $n \in \{0, \dots, \eta\}$, there exist $e_n \in \I$ of level $n$ and, for sufficiently small $h_n \in \Hang_{n-1}$, an analytic continuation $\vv_n:\frU_{h_n} \into \CC$ of $v_n$  such that $|\vv_n(x)| \le 1/e_n(|x|)$ for $x \in \frU_{h_n}$; without loss of generality, we may assume that $\frU_{h_\eta} = \frU$.  Since $\bar\u\rest{\RR} = v_0 \circ \exp + \cdots + v_\eta \circ \exp$ on $\RR$, the identity theorem for holomorphic functions and Remarks \ref{addition_rmk} imply that $$\bar\u = \vv_0 \circ \Exp + \cdots + \vv_\eta \circ \Exp$$ on $\Log(\frU) \cap H(r)$, for some sufficiently large $r \in \RR$; it therefore suffices to prove $(\ast)$ with each $n$ and $u_n := v_n \circ \exp$ in place of $\eta$ and $u$.

So we fix $n \in \{0, \dots, \eta\}$ and a sufficiently small $h \in \Hang_{n-1}$. The case $n=0$ follows from Case 1, so we assume that $n>0$ and $h$ is bounded.  From the assumptions we have
\begin{equation}
\label{eta_bound}
|\bar\u_n(z)| = |(\vv_n \circ \Exp)(z)| \le \frac 1{(e_n \circ \exp)(\re z)} \quad\text{for } z \in \Log(\frU_{h}).
\end{equation}
Set $\bar h:= h \circ \exp$, and let $z \in \Log(\frU_{h/2})$ be such that $$B\left(z,\bar h(\re z)/3\right) \subseteq \Log(\frU_h),$$ which happens by Remark \ref{sr-domain_shrink}(2) as soon as $\re z$ is sufficiently large, and let $w \in B\left(0,\bar h(\re z)/3\right)$.  By the Cauchy estimates and \eqref{eta_bound}, we get
\begin{align*}
\left|(\bar\u_n)^{(i)}(z)\right| &\le \frac{i! \cdot 3^i}{(e_n \circ \exp)\left(\re z-\bar h(\re z)\right) \cdot \bar h(\re z)^i} \\ &\le \frac{i! \cdot 3^i}{(e_n \circ \exp)\left(\re z/2)\right) \cdot \bar h(\re z)^i}.
\end{align*}
From the Taylor series of $\bar\u_n$ at $z$, we therefore get, for $|w| \le \bar h(\re z)/6$, that
\begin{equation*}
|\bar\u_n(z+w) - \bar\u_n(z)| \le |w| \cdot \frac{6}{(e_n \circ \exp)\left(\re z/2)\right) \cdot \bar h(\re z)}.
\end{equation*}
Now note that $\level(h) \le n-2$ by Proposition \ref{alevel_cor}(2), so we have $\level\left(\bar h\right) \le n-1$.  Since $\level(e_n \circ \exp \circ m_{1/2}) = n+1$, it follows that we can take $$\rho:= \frac6{(e_n \circ \exp \circ m_{1/2}) \cdot \bar h} \circ \log,$$ which finishes the proof.
\end{proof}

\begin{prop}
\label{u_small}
Assume that $(\ast)_\eta$ holds.  Then, for sufficiently small $h \in \Hang_{\eta-1}$, the restriction of $\u$ to $\frU_{h}$ is $\D$-Lipschitz.
\end{prop}

\begin{proof}
Let $g \in \Hang_{\eta-1}$ and $\sigma \in \D$ be such that $\frU_g \subseteq \frU$ and \eqref{pre-lipschitz} holds with $g$ in place of $h$.  Set $h:= g/2$ and  $\v:= \u \rest{\frU_h}$, and set $\bar g:= g \circ \exp$ and $\bar h:= h \circ \exp = \bar g/2$.  

Let $x,y \in \frU_h$.  First, if $\arg x = \arg y$ then the horizontal segment $[x,y]$ is contained in $\frU_h$. If also $|x|<|y|$, say, we choose $x = x_0, \dots, x_n = y$ in $\frU_h$ such that $\arg x_0 = \cdots = \arg x_n$ and $d(x_{i+1},x_i) < \bar g(|x_{i}|)/3$, for $i = 0, \dots, n-1$.  Then by \eqref{pre-lipschitz}, we have
\begin{align*}
	d(\v(x),\v(y)) &= |\Log\v(x)-\Log\v(y)| \\ &\le \sum_{i=0}^{n-1} |\Log\v(x_{i+1}) - \Log\v(x_i)|\\ &\le \sum_{i=0}^{n-1} \sigma(|x_i|) \cdot |\log|x_{i+1}| - \log|x_i|| \\ &\le \sigma(|x|) \cdot |\log |x| - \log |y||\\ &= \sigma(|x|) \cdot d(x,y) 
\end{align*}
in this case.
Second, if $|x| = |y|$, then the vertical segment $[x,y]$ is contained in $\frU_h$.  Interpolating along this segment as in the first case, we  obtain
$$d(\v(x),\v(y)) \le \sigma(|x|) \cdot d(x,y)$$ in this case.
In general, there are a vertical segment $S_1 \subseteq \frU_h$ and a horizontal segment $S_2 \subseteq \frU_h$ whose union joins $x$ and $y$.  Since the hypotenuse of a rectangular triangle is at least as long as either of the other sides, we have $$|\log|x| - \log|y|| + |\arg x - \arg y| \le 2d(x,y),$$ so by the above two cases, $\v$ is $\rho$-Lipschitz with $\rho:= 2\sigma$.
\end{proof}

\section{Infinitly increasing germs in $\E$ are expansive and angle-positive}
\label{expansive_section}

We are almost ready to tackle Proposition \ref{exp_ext_prop}.  The next lemma details how the combination of extendability, expansiveness and angle-positivity is preserved under multiplication by a unit.

\begin{lemma}
\label{mult_by_unit}
Let $\eta \in \NN$ and $\f:\frU \into \frV$ be an expansive, angle-positive $(\eta,\eta)$-map.  Let also $\u:\frU \into \LL$ be a holomorphic, $\D$-Lipschitz unit at $\infty$.  Then there exists $r>0$ such that $\f\u \rest{\frU \cap H_{\LL}(r)}$ is an expansive, angle-positive $(\eta,\eta)$-map.
\end{lemma}

\begin{proof}
By Lemma \ref{mult_by_unit_pos_exp}(1), $\f\u$ is expansive and angle-positive, so it remains to show that $\f\uu$ is an $(\eta,\eta)$-map.  The clauses below refer to Definition \ref{weakly_extending}.

To establish Clause (1a) for $\f\u$, we fix $k \ge \eta-1$ and a $k$-domain $\frU' \subseteq \frU$.
First, we show that $\f\u(\frU') \supseteq \frU_h$ for some $h \in \Hang_{k-\eta}$:  since $\f(\frU')$ is a $(k-\eta)$-domain, there exists $h_1 \in \Hang_{k-\eta}$ such that $\frU_{h_1} \subseteq \f(\frU')$; we may assume that $\cl(\frU_{h_1}) \subseteq \f(\frU')$.  Since $\f$ is biholomorphic, the preimage $\frU'':= \f^{-1}(\frU_{h_1})$ is a simply connected domain such that $\partial \frU'' = \f^{-1}(\partial \frU_{h_1})$, where $\partial A$ denotes the topological boundary of $A$.  Since $\f$ maps $(0,\infty)$ into $(0,\infty)$ and $|\f(x)| \to \infty$ as $|x| \to \infty$, it follows that $\frU''$ is a simply connected domain containing all sufficiently large positive reals, mirror symmetric with respect to the real line (by the Schwartz reflection principle) and bounded in $\set{\arg x > 0}$ by the image of the curve $t \mapsto \f^{-1}(t,h_1(t))$.  Since $\f\u$ is also biholomorphic, maps $(0,\infty)$ into $(0,\infty)$ and satisfies $|\f\u(x)| \to \infty$ as $|x| \to \infty$, the same argument gives that $\frV'':= \f\uu(\frU'')$ is such a domain with boundary curve $$t \mapsto \f\u\left(\f^{-1}(t, h_1(t)\right).$$  However, for sufficiently large $x = \f^{-1}(t,h_1(t))$, we get from Lemma  \ref{mult_by_unit_pos_exp}(2) and $2|\f\u(x)| \ge |\f(x)| = t \ge \frac12 |\f\u(x)|$ that
\begin{equation*}
\arg(\f\u(x)) \ge h_1(t)/2 \ge h(|\f\uu(x)|),
\end{equation*}
where $h:= m_{1/2} \circ h_1 \circ m_{c}$ belongs to $\Hang_{k-\eta}$, by Proposition \ref{alevel_cor}, with $c = \frac12$ if $h_1$ is increasing and $c = 2$ if $h_1$ is decreasing.  Arguing similarly for sufficiently large $x = \f^{-1}(t,-h_1(t))$, we conclude that $\frU_h \subseteq \f\u(\frU')$, as required.

Second, we show that $\f\u(\frU') \subseteq \frU_h$ for some $h \in \Hang_{k-\eta}$: since $\f(\frU')$ is a $(k-\eta)$-domain, there exists $h_1 \in \Hang_{k-\eta}$ such that $|\arg\f(x)| \le h_1(|\f(x)|)$ for $x \in \frU'$.  Arguing as above, it follows from Lemma \ref{mult_by_unit_pos_exp}(2) that
\begin{equation*}
|\arg\f\u(x)| \le 2h_1(c|\f\u(x)|)
\end{equation*}
for sufficiently large $x \in \frU'$, where $c = \frac12$ if $h_1$ is decreasing and $c = 2$ if $h_1$ is increasing.  In particular, we have $\f\u(\frU') \subseteq \frU_h$, where $h:= m_2 \circ h_1 \circ m_c$; this $h$ belongs to $\Hang_{k-\eta}$ by Proposition \ref{alevel_cor}, as required.

To establish Clause (1b) for $\f\u$, we let $h \in \Hang_{k-\eta}$ and set $h_1:= m_{1/2} \circ h \circ m_{1/c}$, where $c = \frac12$ if $h$ is decreasing and $c = 2$ if $h$ is increasing.  Then $h_1 \in \Hang_{k-\eta}$ by Proposition \ref{alevel_cor} and, since $\f$ is an $(\eta,\eta)$-map, there exists a $k$-domain $\frU' \subseteq U$ such that $\f(\frU') \subseteq \frU_{h_1}$.  The argument of the previous paragraph now implies that $\f\u(\frU') \subseteq \frU_h$, as required.

Finally, Clause (1c) for $\f\u$ follows from the corresponding clause for $\f$, because $\u$ is a unit at $\infty$.
\end{proof}

\begin{prop}
\label{exp_ext_prop}
Let $\eta \ge 0$ and $f \in \E_\eta \cap \I$ be of level $\lambda \in \{0, \dots, \eta\}$.  Then $f$ has an expansive, angle-positive $(\eta,\lambda)$-extension $\f:\frU \into \frV$.
\end{prop}

\begin{proof}
By induction on the pair $(\eta,\eta-\lambda) \ge (0,0)$, ordered lexicographically.

If $\eta = 0$, then $\lambda = 0$, and we distinguish two cases: if $f \in \M_0$, then $f = p_r$ for some $r>0$, so $f$ has an expansive, angle-positive $(0,0)$-extension on $\LL$ by Examples \ref{levelled_maps} and  \ref{expansive_expl}; in particular, it follows that $(\ast)_0$ holds.  In general, we have $f = amu$ for some $a \in \RR$, $m \in \M_0$ and $u \in \E_0$ a unit at $+\infty$.  By the earlier-mentioned examples and Lemmas \ref{weak_ext_comp} and \ref{comp_expansive}, the function $g:= am$ has an expansive, angle-positive $(0,0)$-extension $\g:\frU \into \frV$.  On the other hand, after shrinking $\frU$ if necessary, we get from Corollary \ref{unit_lemma} and Proposition \ref{u_small} that $u$ has a holomorphic and $\D$-Lipschitz extension $\u:\frU \into \LL$ that is a unit at $\infty$.  It follows from Lemma \ref{mult_by_unit} that $\g\u$ is an expansive, angle-positive $(0,0)$-extension of $f$.

Assume now that $(\eta,\eta-\lambda) > (0,0)$ and that the proposition holds for lower values of $(\eta,\eta-\lambda)$.  If $f \in \M_\eta$, then $\lambda = \eta > 0$ and $f = \exp \circ g$ for some $g \in \E_{\eta-1}$, so $f$ has an expansive, angle-positive $(\eta,\eta)$-extension on $U$ by Examples \ref{levelled_maps} and \ref{expansive_expl}, Lemmas \ref{weak_ext_comp} and \ref{comp_expansive} and the inductive hypothesis; in particular, it follows that $(\ast)_\eta$ holds.  In general, we have $f = amu$ for some $a \in \RR$, $m \in \M_\lambda$ and $u \in \E_\eta$ a unit at $+\infty$.  As before, the product $g:= am$ has an expansive, angle-positive $(\eta,\lambda)$-extension $\g:\frU \into \frV$, and by Corollary \ref{unit_lemma} and Proposition \ref{u_small}, the germ $u$ has a holomorphic and $\D$-Lipschitz extension $\u:\frU \into \LL$ that is a unit at $\infty$.  It follows from Corollary \ref{expansive_unit} that $\g\u$ is expansive, so it remains to show that $\g\u$ is an angle-positive $(\eta,\lambda)$-map.

If $\lambda = \eta$, it follows from Lemma \ref{mult_by_unit} that $\g\uu$ is an angle-positive $(\eta,\eta)$-map, so we assume from now on that $\lambda < \eta$.  By Proposition \ref{exp_comp}, the function $\exp \circ f$ also belongs to $\E_\eta$ and, since it has level $\lambda+1$, the inductive hypothesis gives an expansive, angle-positive $(\eta,\lambda+1)$-extension $\h:\frU \into \frV'$ of $\exp \circ f$.  It follows from the identity theorem for holomorphic functions, Examples \ref{levelled_maps} and \ref{expansive_expl} and Lemmas \ref{weak_ext_comp} and \ref{comp_expansive} that $\g\uu = \llog \circ \h$ is an angle-positive $(\eta,\lambda)$-map.
\end{proof}

\begin{thm}[Continuation]
	\label{I_ext_prop}
	Let $f \in \I$, and set $$\eta:= \max\{0,\eh(f)\} \quad\text{and}\quad \lambda:= \level(f) \le \eta.$$  Then $f$ has an angle-positive $(\eta,\lambda)$-extension $\f:\frU \into \frV$.
\end{thm}

\begin{proof}
	By Proposition \ref{1.11}, there exist $k \in \NN$ and $g \in \E_{\eta+k}$ such that $f = g \circ \log_k$.  By Proposition \ref{exp_ext_prop}, the germ $g$ has an expansive, angle-positive $(\eta+k,\lambda+k)$-extension $\g:\frU' \into \frV'$.  By Examples \ref{levelled_maps} and \ref{expansive_expl} and Lemmas \ref{weak_ext_comp} and \ref{comp_expansive}, $\log_k$ has an angle-positive $(0,-k)$-extension $\llog_k:\frU \into \frU''$; in particular, we may assume that $\frU'' \subseteq \frU'$.  Hence, again by Lemmas \ref{weak_ext_comp} and \ref{comp_expansive}, the map $\g \circ \llog_k:\frU \into \frV:= \g(\frU'')$ is an angle-positive $(\eta,\lambda)$-extension of $f$.
\end{proof}

The following gives the left-to-right implication of the Continuation Corollary:

\begin{cor}
	\label{ltr-implication}
	Let $f \in \H$ and set $\eta:= \max\{0,\eh(f)\}$.  Then there exist an $(\eta-1)$-domain $\frU$ and a half-bounded, analytic continuation $\f:\frU \into \LL$ of $f$.
\end{cor}

\begin{proof}
	If $f \in \I$, the conclusion follows from the Continuation Theorem, so we assume from now on that $f \notin \I$.  Then there exists $c \in \RR$ such that $\frac1{f-c} \in \I$.  By the Continuation Theorem, there is an $(\eta,\lambda)$-extension $\g:\frU \into \frV$, where $\lambda:= \level\left(\frac1{f-c}\right)$.  Therefore, $1/\g:\frU \into \LL$ is a half-bounded, analytic continuation of $f-c$; in particular, we are done if $c=0$.  So assume also that $c \ne 0$, and let $r>0$ be such that $|1/\g(x)| < |c|/2$ for $x \in \frU \cap H_\LL(r)$.  If $c>0$, we set $\sigma_c:= \pi_0$, while if $c<0$, we let $\sigma_c$ be the restriction of $\pi$ to $\set{x \in \LL:\ 0 < \arg x < 2\pi}$.  Then the map $\f:\frU \cap H_\LL(r) \into \LL$ defined by $$\f(x):= \sigma_c^{-1}\left(\left(\pi \circ \frac1{\g}\right)(x) + c\right)$$ is a half-bounded, analytic continuation of $f$.
\end{proof}

\subsection*{Complex continuation}
For $\eta \ge -1$, we call a set $U \subseteq \CC$ an \textbf{$\eta$-domain} if there exists an $\eta$-domain $\frU \subseteq \LL$ such that $U = \pi(\frU)$.  For example, $\CC \setminus (-\infty,0]$ as well as each half-plane $H(a)$ are 0-domains, for $a \ge 0$, while it follows from Example \ref{al_expls}(4) that the strip $S = \set{z \in \CC:\ \re z > 0,\ |\im z| < \pi}$ is a 1-domain.  

What can be said about complex analytic extensions is exemplified by the germs $\exp$ and $\log$: $\exp$ has a non-injective analytic extension $\bexp$ on the 0-domain $\CC \setminus (-\infty,0]$.  Its image is the $(-1)$-domain $\pi(\U)$, where $\U = \eexp(S_\LL(\pi))$; in particular, the extension $\bexp$ is half-bounded.  The Continuation Theorem \ref{I_ext_prop} also implies that $\bexp$ maps $k$-domains contained in $\CC \setminus (-\infty,0]$ to $(k-1)$-domains, for $k \ge 0$, that is, it maps $k$-domains into $(k-\level(\exp))$-domains.  Moreover, it maps points with positive imaginary parts in $S$ to points with positive imaginary parts.  Correspondingly, the analytic extension $\blog$ of $\log$ to $\CC\setminus((-\infty,0] \cup \bar B(0,1))$ maps $k$-domains into $(k-\level(\log))$-domains.

\begin{nrmk}
	\label{eta-domain-rmk}
	Because $\eta$-domains in $\LL$ are connected and each of the lines $\{\arg x = \pi\}$ and $\{\arg x = -\pi\}$ topologically separates $\LL$, a set $U \subseteq \CC \setminus (-\infty,0]$ is an $\eta$-domain if and only if there is an $\eta$-domain $\frU \subseteq S_\LL(\pi)$ such that $U = \pi_0(\frU)$.  Thus, if $U \subseteq \CC$ is an $\eta$-domain, then $\eta > 0$ implies $U \subseteq \CC \setminus (-\infty,0]$, while $U \subseteq \CC \setminus (-\infty,0]$ implies $\eta \ge 0$.
\end{nrmk} 

\begin{cor}[Complex Continuation]
	\label{complex_cont_cor}
	Let $f \in \H$, set $\eta:= \max\{1,\eh(f)\}$ and $\lambda:= \level(f)$.
	\begin{enumerate}
		\item There are an $(\eta-1)$-domain $U \subseteq \CC \setminus (-\infty,0]$ and a half-bounded complex analytic continuation $\ff:U \into \CC$ of $f$.   
		\item If $f \in \I$, there are an $(\eta-1)$-domain $U \subseteq \CC \setminus (-\infty,0]$, an $(\eta-\lambda-1)$-domain $V \subseteq \CC$ and a complex analytic continuation $\ff:U \into V$ of $f$ such that 
		\begin{enumerate}
			\item $|\ff(z)| \to \infty$ as $|z| \to \infty$, for $z \in U$;
			\item for every $k$-domain $W \subseteq U$, the image $\ff(W)$ is a $(k-\lambda)$-domain, and if $\ff(W) \subseteq \CC \setminus (-\infty,0]$, then the restriction $\ff\rest{W}:W \into \ff(W)$ is biholomorphic and we have
			$$\sign(\arg\ff(z)) = \sign(\arg z) = \sign(\im z) =
			\sign(\im\ff(z))$$ for $z \in W$;
			\item if $f \prec p_2$, then $\ff(U \cap H(0)) \subseteq \CC \setminus (-\infty,0]$.
		\end{enumerate}
	\end{enumerate}
	In each of these situations, if $\,\eh(f) \le 0$, we can choose $U$ such that $H(a) \subseteq U$ for some $a \ge 0$.
\end{cor}

\begin{proof}
	(1) Let $\f:\frU \into \LL$ be an $\LL$-analytic continuation of $f$ obtained from Continuation Corollary \ref{H_ext_cor}, and define $U:= \pi_0(\frU)$ and $\ff:U \into \CC$ by $$\ff(z):= \pi(\f(\pi_0^{-1}(z))).$$  
	
	(2) Let $\f:\frU \into \frV$ be an $\LL$-analytic continuation of $f$ obtained from the Simplified Continuation Theorem \ref{simplified_ext}, and define $U:= \pi_0(\frU)$, $V:= \pi(\frV)$ and $\ff:U \into V$ by $$\ff(z):= \pi(\f(\pi_0^{-1}(z))).$$   Since $\f$ is biholomorphic, after replacing $\frU$ by an $(\eta-1)$-domain $\frV$ such that $\cl(\frV) \subseteq \frU$ if necessary, the continuity of $f^{-1}$ on the closure of $\frU$ implies that preimages under $\f$ of compact sets are compact; in particular, $|\f(x)| \to \infty$ as $|x| \to \infty$, for $x \in \frU$, which implies (a).  Second, (b) follows from the facts that $\f$ is biholomorphic, maps $k$-domains to $(k-\lambda)$-domains and is angle-positive, and from Remark \ref{eta-domain-rmk}.  
	
	Third, assume that $f \prec p_2$, so that $h:= p_{2}/f \in \I$.  Let $\h:\frU' \into \frV'$ be an $\LL$-analytic continuation of $h$ obtained from the Simplified Continuation Theorem \ref{simplified_ext}.  Since $\eh(h) \le 0$, we may assume that $\frU' = \frU$.  Since $\h$ is angle-positive we have, for $x \in \frU$ with $\arg x > 0$, that
	\begin{equation*}
		0 < \arg\h(x) = 2\arg x - \arg\f(x).
	\end{equation*}
	Since $\f$ is also angle-positive, it follows that $0 < \arg\f(x) < 2\arg x$ for such $x \in \frU$.  Arguing similarly for $x \in \frU$ with $\arg x < 0$, we obtain 
	\begin{equation*}
		|\arg\f(x)| < 2 |\arg x| \quad\text{for } x \in \frU;
	\end{equation*}
	in particular, we have $\f(\frU \cap S_\LL(\pi/2)) \subseteq S_\LL(\pi)$, that is, $\ff(U \cap H(0)) \subseteq \CC \setminus (-\infty,0]$, which proves (c).
	
	Finally, if $\eh(f) \le 0$, then $\frU$ is a $(-1)$-domain; in particular, $\frU$ is not angle-bounded, so that $H(a) \subseteq U$ for some $a \ge 0$.
\end{proof}

\section{Upper bounds on exponential height and simple germs}
\label{geom_simple_section}

In this section, we consider the following question: given germs $f,g \in \I$, what can be said about $\eh\left(f \circ g^{-1}\right)$ in terms of $\eh(f)$ and $\eh(g)$?  

\begin{expl}
	\label{question_expl}
	Let $f,g \in \I$ be such that $\eh(f), \eh(g) \le 0$ and $g > f$.  Then $x > f \circ g^{-1}$ and $\lambda:= \level(f \circ g^{-1}) \le 0$, but we do not know $\eh(f \circ g^{-1})$.  Nevertheless, by the Continuation Theorem, there are a $(0,\level(f))$-extension $\f:\frU_1 \into \frV_1$ of $f$ and a $(0,\level(g))$-extension $\g:\frU_2 \into \frV_2$; by the definition of angular level for maps, we may assume that $\frU_2 \subseteq \frU_1$.  Hence $\f \circ \g^{-1}:\frV_2 \into \frV_1$ is analytic continuation of $f \circ g^{-1}$.  Note, however, that $\frV_2$ can be arbitrarily small, even if $\level(f) = \level(g)$; moreover, we do not know if $\f \circ \g^{-1}$ is an $(\eta,\lambda)$-extension of $f \circ g^{-1}$, for some appropriate $\eta$.
\end{expl}

The goal of this section is to get a better result than in the previous example, at least for certain germs in $\I$.  To get a handle on compositional inversion, we note the following uniqueness principle:

\begin{prop}
\label{P-L}
Let $k \ge -1$, $\frU \subseteq \LL$ be a $k$-domain and $\ff:\frU \into \CC$ be a bounded holomorphic function.  Assume that there exists $h \in \D$ of level $k+2$ such that $$|\ff(x)| = O(h(x)) \quad\text{as } x \to \infty \text{ in } \RR.$$  Then $\ff = 0$.
\end{prop}

\begin{proof}
Replacing $\frU$ by a smaller $k$-domain if necessary, we may assume that $\frU = \llog_{k+1}(\frV)$ for some $(-1)$-domain $\frV \subseteq \LL$.  Then, replacing $\frU$ by $\frV$, $\ff$ by $\ff \circ \llog_{k+1}$ and $h$ by $h \circ \log_{k+1}$ if necessary, we may assume that $k=-1$.

Next, let $l \in \NN$ and $g \in \E$ of level $l+1$ be such that $h = g \circ \log_l$, and let $\nu \in \NN$ be nonzero such that $|g| \le 1/(\exp_{l+1} \circ p_{1/\nu})$.  Then $$\left|h \circ \exp_l \circ p_\nu \circ \log_l \circ p_2\right| \le 1/(\exp \circ p_2);$$ replacing $\frU$ by the $(-1)$-domain $p_{1/2}(\exp_l(p_{1/\nu}(\log_l(\frU))))$ and $\ff$ by $\ff \circ \eexp_l \circ \p_\nu \circ \llog_l \circ \p_2$, we may assume that $h = 1/(\exp \circ p_2)$.

Now the restriction of $\ff$ to $\pi_0^{-1}(H(a))$, for a sufficiently large $a \ge 0$, is bounded, holomorphic and satisfies $|\ff(x)| = o(\exp(-nx))$ as $x \to +\infty$ in $\RR$, for all $n \in \NN$.  Therefore, $\ff = 0$ by \cite[Lemma 24.37]{Ilyashenko:2008fk} and the identity theorem for holomorphic functions.
\end{proof}

From Proposition \ref{P-L} and the Continuation Theorem, we obtain the right-to-left implication of the Continuation Corollary:  

\begin{prop}
\label{max_eh}
Let $f \in \H$ and $\eta \in \NN$, and assume that $f$ has a half-bounded analytic continuation $\f:\frU \into \KK$, where $\frU \subseteq \LL$ is an $(\eta-1)$-domain and $\KK = \CC$ or $\KK = \LL$.  Then $\eh(f) \le \eta$.
\end{prop}

\begin{proof}
Replacing $f$ by $1/f$ if necessary, we may assume that $f$ is bounded; in particular, every principle monomial of $f$ is either a constant or small.  Hence
by Lemma \ref{1.5}, there exist $a \in \RR$ and $f_1, f_2 \in \D$ such that $\eh(f_1) \le \eta$ and every $m \in \M(f_2)$ is small and satisfies $\eh(m) > \eta$, and such that $f = a + f_1 + f_2$.  We claim that $f_2 = 0$, which then proves the proposition.  

To see the claim, by Corollary \ref{eh-cor}(3), the germ $1/f_1 \in \I$ has exponential height at most $\eta$, so by the Continuation Theorem $1/f_1$ is $(\eta,\level(f_1))$-extendable.  So after shrinking $\frU$ if necessary, $f_1$ has a bounded analytic continuation $\f_1:\frU \into \LL$.  Therefore, $f_2$  has a bounded analytic continuation $\ff_2:\frU \into \CC$ defined by $$\ff_2:= \begin{cases}
	\pi \circ \f - a - \pi \circ \f_1 &\text{if } \KK = \LL, \\ \f - a - \pi \circ \f_1 &\text{if } \KK = \CC;
\end{cases}$$ on the other hand, since $\eh(m) \ge \eta+1$ for every principle monomial of $f_2$, we have $\level(f_2) \ge \eta+1$.  Hence $\f_2 = 0$ by Proposition \ref{P-L}; in particular, $f_2 = 0$, as claimed.
\end{proof}

Recall from Corollary \ref{eh-cor}(3) that $\H_{\le 0} = \set{f \in \H:\ \eh(f) \le 0}$ is a Hardy field.

\begin{cor}
	\label{comp_inequality}
	The Hardy field $\H_{\le 0}$ is stable under composition; that is, given $f,g \in \H_{\le 0}$ such that $g \in \I$, we have $f \circ g \in \H_{\le 0}$.
\end{cor}

\begin{proof}
	Assume first that $f \in \I$, and let $k:= \level(f)$ and $l:= \level(g)$.  By the Continuation Theorem, $f$ has a $(0,k)$-extension $\f:\frU_1 \into \frV_1$ and $g$ has a $(0,l)$-extension $\g:\frU_2 \into \frV_2$.  Condition (b) of the definition of $(0,l)$-map implies that there exists a $(-1)$-domain $\frU_2' \subseteq \frU_2$ such that $\g(\frU'_2) \subseteq \frU_1$.  Hence $\f \circ \g:\frU'_2 \into \LL$ is a half-bounded analytic continuation of $f \circ g$, so $\eh(f \circ g) \le 0$ by Proposition \ref{max_eh} and Corollary \ref{eh-cor}(3).
	
	If $f \notin \I$, then there exists $c \in \RR$ such that $\frac1{f-c} \in \I$, so the corollary follows from the previous case and Corollary \ref{eh-cor}(3).
\end{proof}

\begin{cor}
\label{eh_of_inverse}
Let $f \in \I$, set $\eta:=\max\{\eh(f),0\}$ and $\lambda:= \level(f) \le \eta$.  Then $$\eh\left(f^{-1}\right) \le \eta-\lambda;$$ in particular, $f^{-1}$ is $(\eta-\lambda,-\lambda)$-extendable.
\end{cor}

\begin{rmk}
	Note that the compositional inverse of an $(\eta,\lambda)$-map is not \textit{a priori} an $(\eta-\lambda,-\lambda)$-map.
\end{rmk}

\begin{proof}
Let $\f:\frU \into \frV$ be an $(\eta,\lambda)$-extension of $f$.  Since $\f$ is biholomorphic, the germ $f^{-1}$ has an analytic continuation $\f^{-1}:\frV \into \frU$; since $\frV$ is an $(\eta-\lambda-1)$-domain, it follows from Proposition \ref{max_eh} and Corollary \ref{eh-cor}(3) that $\eh\left(f^{-1}\right) \le \eta-\lambda$.
\end{proof}

\begin{df}
	\label{simple_df}
	We call $f \in \H$ \textbf{simple} if $\level(f) =\eh(f)$. 
\end{df}

\begin{expl}
	If $f \in \H$ is purely infinite then, by definition, we have $\eh(f) = \eh(\lm(f)) = \level(f)$, that is, $f$ is simple.
\end{expl}

\begin{lemma}
\label{gp_comp}
Let $f \in \I$ be simple.  Then $f \circ \exp$ and $f \circ \log$ are simple.
\end{lemma}

\begin{proof}
Corollary \ref{eh-cor}(1) and Fact \ref{level_facts}(3).
%
\end{proof}

We denote by $\S$ the set of all simple $f \in \I$ and set $$\S_0:= \set{f \in \S:\ \level(f) = 0}.$$

\begin{prop}
\label{gp_0_group}
The set $\S_0$ is a compositional subgroup of $\I$. 
\end{prop}

\begin{proof}
Let $f,g \in \S_0$.  Then $\eh\left(f^{-1}\right) \ge \level\left(f^{-1}\right) = 0$ and, by Corollary \ref{eh_of_inverse}, we also have $\eh\left(f^{-1}\right) \le 0$, so that $f^{-1} \in \S_0$.  Moreover, we have $$0 = \level(f \circ g) \le \eh(f \circ g) \le 0$$ by Corollary \ref{eh_of_inverse}, so $f \circ g \in \S_0$.
\end{proof}

\begin{cor}
\label{gp_ext}
Let $f \in \H$, $g \in \I$ and set $\eta:= \eh(f) \ge \lambda:= \level(f)$ and $k:= \eh(g) \ge l:= \level(g)$.  Then
\begin{enumerate}
\item $\eh\left(f \circ g^{-1}\right) \le \max\{\eta+k-2l,k-l\}$;
\item if $f,g \in \S$ and $\eta \ge k$, then $f \circ g^{-1} \in \S$;
\item if $g \in \S$ and $k \le 0$, then $g^{-1} \in \S$.
\end{enumerate}
\end{cor}

\begin{proof}
(1) By Proposition \ref{exp_comp}, the germ $\exp_{k-l} \circ g$ is simple; since $$\eh\left(f \circ g^{-1}\right) = \eh\left(f \circ g^{-1} \circ \log_{k-l}\right) + k-l,$$ we may assume that $g$ is simple and show that $$\eh\left(f \circ g^{-1}\right) \le \max\{\eta-l,0\}$$ in this case.  Since $g$ is simple, the germ $g \circ \log_k$ is simple by Lemma \ref{gp_comp}, and since $$f \circ g^{-1} = (f \circ \log_k) \circ (g \circ \log_k)^{-1},$$ we may even assume that $k=l = 0$.  But then $g^{-1} \in \S_0$ by Proposition \ref{gp_0_group}, so $g^{-1}$ is $(0,0)$-extendable, and we set $\eta':= \max\{\eta,0\}$.  

Next, assume that $f$ is large.  Since $\eta' = \max\{\eta,0\}$, we get from the Continuation Theorem that $f$ or $-f$ is $(\eta',\lambda)$-extendable; it follows from Lemma \ref{weak_ext_comp} that $f \circ g^{-1}$ or $-f \circ g^{-1}$ is $(\eta',\lambda)$-extendable as well.   Corollaries \ref{eh_of_inverse} and \ref{F_isa_field} now imply that $\eh(f \circ g^{-1}) \le \eta'$, as required.  

Finally, if $f$ is bounded, there exists $c \in \RR$ such that $\frac1{f-c}$ is large, so the previous subcase and Corollary \ref{F_isa_field} also give the claimed conclusion in this subcase.

Part (2) follows from part (1), since $\level\left(f  \circ g^{-1}\right) = \eta-k = \eta'$ in this case, and part (3) follows from part (2) with $f = p_1$.
\end{proof}

\begin{expl}
	\label{question_expl_2}
	In the setting of Example \ref{question_expl}, assume in addition that $f$ and $g$ are simple.  Then we get from Corollary \ref{gp_ext}(1) that $\eh\left(f \circ g^{-1}\right) \le 0$, so by the Continuation Theorem, the germ $f \circ g^{-1}$ is $(0,\level(f)-\level(g))$-extendable.  This holds irrespective of what $\level(g)$ actually is, which represents a big improvement over the  observation in Example \ref{question_expl}.
\end{expl}

 \section{Definable analytic continuations} 
\label{definability_section}

In this section, we study the definability of the analytic continuations obtained from the Continuation Theorem.  We call a set $\frS \subseteq \LL$ \textbf{angle-bounded} if there exists $K>0$ such that $\frS \subseteq S_\LL(K)$.

\begin{rmk}
	If $\frA,\frB \subseteq \LL$ are angle-bounded, then the sets $\frA\frB$ and $\frA/\frB$ are angle-bounded.
\end{rmk}

\begin{lemma}
\label{laurent_dfbl_ext}
	Let $k \in \NN$ and $P \in \Pc{R}{X_1, \dots, X_k}$, and let $r>0$ and $\bP:B_\LL(r)^{k} \into \CC$ be a sum of $P$.  Let $\frU_1, \dots, \frU_{k} \subseteq B_\LL(r)$ be definable domains and set $\frU:= \frU_1 \times \cdots \times \frU_{k}$.  Then $\bP\rest{\frU}$ is definable if and only if each $\frU_i$ is angle-bounded.
\end{lemma}

\begin{proof}
	This follows from the fact that $\bP$ is periodic in each variable.
\end{proof}

The next proposition implies the Definability Theorem.

\begin{prop}
\label{dfbl_ext_lemma}
	Let $f \in \I$, set $\eta:= \max\{0,\eh(f)\}$ and $\lambda:= \level(f)$, and let $\f:\frU \into \frV$ be an $(\eta,\lambda)$-extension of $f$.  Let also $\frU' \subseteq \frU$ be a definable domain.  If $\f(\frU')$ is angle-bounded, then $\f\rest{\frU'}$ is definable.
\end{prop}

\begin{rmk}
	The converse does not hold for $f = \exp$, as $\eexp$ is definable but has angle-unbounded image.
\end{rmk}

\begin{proof}
	Note that the conclusion holds if $f = \log$, because $\llog$ is definable.  In general, there are $k \in \NN$ and $g \in \E_{\eh(f)+k}$ such that $f = g \circ \log_k$.  We claim that if the proposition holds with $g$ in place of $f$, then it holds for $f$.
	
	To see the claim, assume the proposition holds with $g$ in place of $f$, and let $\g:\frU_1 \into \frV_1$ be an $(\eh(f)+k, \lambda+k)$-extension of $g$.  Intersecting $\frU'$ with $H_\LL(R)$ for some sufficiently large $R>0$, we may assume that $$\frU'':= \llog_k(\frU') \subseteq H_\LL(1);$$  in particular, $\llog_k\rest{\frU'}$ and $\frU''$ are definable.  Now, since $\f(\frU')$ is angle-bounded, we have $$\g(\frU'') = \g(\llog_k(\frU')) = \f(\frU')$$ is angle-bounded.  So by hypothesis, the restriction $\g\rest{\frU''}$ is definable; but then $$\f\rest{\frU'} = (\g \circ \llog_k)\rest{\frU'} = \left(\g\rest{\frU''}\right) \circ \left(\llog_k\rest{\frU'}\right)$$ is definable as well, which proves the claim.
	
	By the claim, we may assume that $f \in \E$ and, in this case we prove, by induction on $\eta = \eh(f)$, that 
	\begin{itemize}
		\item [$(\ast)$] if $\f(\frU')$ is angle-bounded, then $\frU'$ is angle-bounded and $\f\rest{\frU'}$ is definable.
	\end{itemize}
	
	\subsection*{Case $\eta = 0$:} If $f \in \M_0$, then $f = p_k$ for some $k \in \ZZ$, which has definable $(0,0)$-extension $\p_k:\LL \into \LL$.  The same goes for $f = m_a$ with $a \in \RR$, which has definable $(0,0)$-extension $\m_a:\LL \into \LL$.  Hence $\m_a \circ \p_k$ is a definable analytic continuation of $m_a \circ p_k$, and both the image and preimage under $\m_a \circ \p_k$ of any angle-bounded domain are angle-bounded. 
	
	On the other hand, if $f$ is a unit at $+\infty$, then $f = 1 + g$ with $g \in \E_0$ small.  By Lemma \ref{f_going_to_0}(1), there is a $P \in \Pc{R}{T}$ such that $g(x) = P(1/x)$ for sufficiently large $x>0$.  Let $r>0$ be such that $|P(z)| < 1/2$ for $|z| < r$, and let $\bP:B_\LL(r) \into \CC$ be the corresponding analytic continuation of $P$.  Then $$x \mapsto \f(x):= \pi_0^{-1}(1+\bP(1/x)):H_\LL(1/r) \into \LL$$ is an analytic continuation of $f$ that has image contained in $S_\LL(\pi/2)$ and, by Lemma \ref{laurent_dfbl_ext}, the restriction of $\f$ to any definable, angle-bounded domain is definable.
	
	In general, $f = (m_a \circ p_k) u$ for some nonzero $a \in \RR$, $k \in \ZZ$ and $u \in \E_0$ such that $u$ is a unit at $+\infty$.  Assume that $\f(\frU')$ is angle-bounded; since the analytic continuation $\u$ of $u$ has image contained in $S_\LL(\pi/2)$, and since $(\m_a \circ \p_k)(\frU') \subseteq \f(\frU') / \u(\frU')$, it follows that $(\m_a \circ \p_k)(\frU')$ is angle-bounded as well.  So by the above, $(\m_a \circ \p_k)\rest{\frU'}$ is definable and $\frU'$ is angle-bounded, so that $\u\rest{\frU'}$ and hence $\f\rest{\frU'}$ is definable as well.
	
	\subsection*{Case $\eta>0$:} Assume that $(\ast)$ holds for $g \in \E_{\eta-1}$ and that $\f(\frU')$ is angle-bounded.  If $f \in \M_\eta$, then $f = \exp \circ g$ for some $g \in \E_{\eta-1}$; let $\g:\frU_1 \into \frV_1$ be an $(\eta-1,\lambda-1)$-extension of $g$.  Since $\frU_1$ is an $(\eta-2)$-domain, we may assume that $\frU' \subseteq \frU_1$; and  since $|\f(z)| \to \infty$ uniformly in $|z|$, we may assume that $\f(\frU') \subseteq H_\LL(1)$.  Then $$\g(\frU') = \llog(\f(\frU')) \subseteq \llog(H_\LL(1)) = S_\LL(\pi/2);$$ so that $\frU'$ is angle-bounded and $\g\rest{\frU'}$ is definable, by the inductive hypothesis, and $\eexp\rest{\g(\frU')}$ is definable.  Hence $\f\rest{\frU'} = \eexp\rest{\g(\frU')} \circ \g\rest{\frU'}$ is definable as well, as claimed.
	
	On the other hand, if $f$ is a unit at $+\infty$, then $f = 1 + g$ with $g \in \E_\eta$ small.  By Lemma \ref{f_going_to_0}(1), there are $k \in \NN$, small $m_1, \dots, m_k \in \M_\eta$, an $(\eta-1)$-domain $\frU_1$ and a $P \in \Pc{R}{X_1, \dots, X_k}$ such that $\gg(x) = \bP(\m_1(x), \dots, \m_k(x))$ for sufficiently large $x \in \frU_1$, where $\g:\frU_1 \into \CC$ and $\m_1, \dots, \m_k:\frU_1 \into \LL$ are analytic continuations of $g$ and $m_1, \dots, m_k$, respectively, and $\bP:B_\LL(r) \into \CC$ be the corresponding analytic continuation of $P$, for some sufficiently small $r>0$.  Then $$x \mapsto \f(x):= \pi_0^{-1}(1+\bP(\m_1(x), \dots, \m_k(x))):H_\LL(1/r) \into \LL$$ is an analytic continuation of $f$ that has image contained in $S_\LL(\pi/2)$ and, by Lemma \ref{laurent_dfbl_ext} and the previous paragraph, if each $\m_i(\frU')$ is angle-bounded, then $\f\rest{\frU'}$ is definable as well.
		
	In general, $f = (m_a \circ n) u$ for some nonzero $a \in \RR$, $n \in \M_\eta$ and $u \in \E_\eta$ such that $u$ is a unit at $+\infty$.  Assume that $\f(\frU')$ is angle-bounded; since the analytic continuation $\u$ of $u$ has image contained in $S_\LL(\pi/2)$, and since $(\m_a \circ \n)(\frU') \subseteq \f(\frU') / \u(\frU')$, it follows that $(\m_a \circ \n)(\frU')$ is angle-bounded as well.  So by the above, $(\m_a \circ \n)\rest{\frU'}$ is definable and $\frU'$ is angle-bounded, so that $\u\rest{\frU'}$ is definable as well.
\end{proof}

As an application of this proposition, we obtain a variant of Wilkie's theorem \cite[Theorem 1.11]{MR3509953} on definable complex continuations: set $$\Hpoly:= \set{f \in \H:\ |f| \le p_n \text{ for some } n \in \NN}$$ and $$\Rpoly:= \set{f \in \H:\ m,\frac1m \in \Hpoly \text{ for every } m \in \M(f)}.$$  Then $\Hpoly$ is a convex subring of $\H$ with maximal ideal $$\Mpoly := \set{f \in \H:\ |f| < p_n \text{ for every } n \in \NN}.$$  Consequently, $\Rpoly$ is a subring of $\Hpoly$.  Moreover, if $P \in \Pc{R}{T}$ and $f \in \Rpoly$ is small, then $P \circ f \in \Rpoly$; it follows that $\Rpoly$ is a field.  Finally, it follows from (E4) and Proposition \ref{1.11} that $$\Hpoly = \Rpoly \oplus \Mpoly$$  as $\RR$-vector spaces;  in particular, $\Rpoly$ is isomorphic to $\Hpoly / \Mpoly$.  It follows from Kuhlmann \cite[Theorem 6.46]{MR1760173} that $\Rpoly$ is the same as the subfield with the same name in Wilkie \cite{MR3509953}.

\begin{rmk}
	Let $f \in \Rpoly$.  Then every $m \in \M(f)$ has comparability class equal to or slower than that of the identity function.  It follows that $\eh(f) \le 0$ so, by the Continuation Corollary, there exist a $(-1)$-domain $U_f$ and a half-bounded, analytic continuation $\f:\frU \into \LL$.  
\end{rmk}

To characterize the germs in $\Rpoly$ in terms of \textit{definable} analytic continuations, we need another lemma.

\begin{lemma}
	\label{cc_for_eh}
	Let $f \in \I$ be such that $\eh(f) \le 0$.  Then $\f$ maps angle-bounded domains into angle-bounded domains if and only if the comparability class of $f$ is equal to or slower than that of the identity map.
\end{lemma}

\begin{proof}
	Assume first that the comparability class of $f$ is equal to or slower than that of the identity map.  Then $f \le p_n$ for some $n \in \NN$, so that $\frac{p_{n+1}}{f} \in \I$.  Since $\eh\left(\frac{p_{n+1}}{f}\right) = \eh(f) \le 0$, there are a $(-1)$-domain $\frU$ and angle-positive, analytic continuations $\f,\g:\frU \into \LL$ of $f$ and $\frac{p_{n+1}}{f}$, respectively.  Thus, for $x \in \frU$ with $\arg x > 0$, we have $$0 < \arg\g(x) = (n+1)\arg x - \arg\f(x),$$ that is, $0 < \arg\f(x) \le (n+1)\arg x$, which proves one direction.
	
	Conversely, assume that the comparability class of the identity map is strictly slower than that of $f$, and fix an angle-positive, analytic continuation $\f:\frU \into \LL$ of $f$.  Let $n \in \NN$; then $\frac{f}{p_{n}} \in \I$, so there are a $(-1)$-domain $\frU_n \subseteq \frU$ and an angle-positive, analytic continuation $\g_n:\frU_n \into \LL$ of $\frac{f}{p_{n}}$.  Fix $c > 0$, and let $x \in \frU_n$ be such that $\arg x = c$; then $$0 < \arg\g_n(x) = \arg\f(x) - n \arg x = \arg\f(x) -nc,$$ so that $\arg\f(x) > nc$.  Since $n \in \NN$ was arbitrary, this means that we can find $x_n \in \frU$, for $n \in \NN$, such that $\arg x_n = c$ and $\arg\f(x_n) \to +\infty$ as $n \to \infty$, which proves the other direction.
\end{proof}

\begin{proof}[Proof of Application 2]
	First assume that $f \in \Rpoly$.  Then $\eh(f) \le 0$ so, by the Continuation Corollary, there are a $(-1)$-domain $\frU$ and a half-bounded, analytic continuation $\f:\frU \into \LL$.  Also, the comparability class of $f$ is equal to or slower than that of the identity function,  so by Lemma \ref{cc_for_eh}, $\f$ maps angle-bounded domains into angle-bounded domains.  Thus, if $\frU' \subseteq \frU$ is an angle-bounded, definable domain, then $\f(\frU')$ is angle-bounded, so by Proposition \ref{dfbl_ext_lemma}, the restriction of $\f:= \pi \circ \f$ to $\frU'$ is definable.
	
	Conversely, assume that there exist a $(-1)$-domain $\frU$ and a half-bounded, analytic continuation $\ff:\frU \into \CC$ of $f$ such that, for every angle-bounded, definable domain $\frU' \subseteq \frU$, the restriction $\ff\rest{\frU'}$ is definable.  Replacing $f$ by $f-c$ for some appropriate $c \in \RR$ if necessary, we may assume that $f$ is small or large; again replacing $f$ by $1/f$ if ncecessary, we may assume that $f$ is small; finally, replacing $f$ by $-f$ if necessary, we may also assume that $f > 0$.  Note that, in this case, $\ff$ is bounded.
	
	Since $f \in \Hpoly$, there are small $f_1 \in \Rpoly$ and $f_2 \in \Mpoly$ such that $f = f_1 + f_2$; we need to show that $f_2 = 0$.  By the left-to-right implication, after shrinking $\frU$ if necessary, there exists a bounded, analytic continuation $\ff_1:\frU \into \CC$ of $f_1$ such that, for every angle-bounded, definable domain $\frU' \subseteq \frU$, the restriction $\ff_1\rest{\frU'}$ is definable.  Therefore, $$\ff_2:= \ff - \ff_1:\frU \into \CC$$ is a bounded, analytic continuation of $f_2$.  Since $\frU$ is a $(-1)$-domain, it follows from Proposition \ref{max_eh} that $\eh(f_2) \le 0$; so by the Continuation Corollary and after shrinking $\frU$ again if necessary, $f_2$ has a half-bounded, analytic continuation $\f_2:\frU \into \LL$.
	
	Assume now, for a contradiction, that $f_2 \ne 0$.  Then by definition, the comparability class of the identity function is stictly slower than that of $f_2$, hence that of $1/f_2$.  By Lemma \ref{cc_for_eh}, there is an angle-bounded domain $\frU' \subseteq U$ such that $\f_2(\frU')$ is not angle-bounded; in particular, $(\pi \circ \f_2)\rest{\frU'}$ is not definable.  But $\pi \circ \f_2 = \ff_2$ by the identity theorem for holomorphic functions, so that $\ff_2\rest{\frU'}$ is not definable.  This contradicts the assumption that both $\ff\rest{\frU'}$ and $\ff_1\rest{\frU'}$ are definable.
\end{proof}


\end{document}